\documentclass{amsart}
\newcommand{\abs}[1]{\lvert#1\rvert}

\usepackage{amssymb}
\usepackage{amsmath}
\usepackage{amsfonts}
\usepackage{tikz-cd}
\usepackage{amsthm}
\usepackage{enumerate}
\usepackage{mathrsfs}
\usepackage{todonotes}
\usepackage{relsize}
\usepackage{appendix}
\usetikzlibrary{decorations.pathmorphing}

\usepackage[utf8]{inputenc}

\newtheorem{thm}{Theorem}[section]
\newtheorem{lemma}[thm]{Lemma}

\newtheorem{cor}[thm]{Corollary}
\newtheorem{prop}[thm]{Proposition}

\theoremstyle{definition}
\newtheorem{example}[thm]{Example}
\newtheorem{define}[thm]{Definition}
\newtheorem{convention}[thm]{Convention}

\theoremstyle{remark}
\newtheorem{remark}[thm]{Remark}

\DeclareMathOperator{\Hom}{Hom}

\newcommand{\D}{\mathscr{D}}
\newcommand{\bigO}{\mathscr{O}}
\DeclareMathOperator{\Der}{Der}

\DeclareMathOperator{\ann}{ann}
\DeclareMathOperator{\Sp}{\text{Sp}}
\DeclareMathOperator{\gr}{gr}

\DeclareMathOperator{\A}{A_{\emph{n}}(\mathbb{C})}
\DeclareMathOperator{\V}{V}
\DeclareMathOperator{\x}{\mathfrak{x}}
\DeclareMathOperator{\y}{\mathfrak{y}}
\DeclareMathOperator{\lcm}{lcm}
\DeclareMathOperator{\red}{red}
\DeclareMathOperator{\ad}{ad}
\DeclareMathOperator{\Ad}{Ad}
\DeclareMathOperator{\Lie}{Lie}
\DeclareMathOperator{\tr}{tr}
\DeclareMathOperator{\Div}{div}
\DeclareMathOperator{\degree}{deg}
\DeclareMathOperator{\mdr}{mdr}



\allowdisplaybreaks

\numberwithin{equation}{section}

\begin{document}
\-

\title[Bernstein--Sato ideals for tame and free arrangements]{Combinatorially determined zeroes of Bernstein--Sato ideals for tame and free arrangements}

\author{Daniel Bath}
\address{Department of Mathematics, Purdue University, West Lafayette, IN, USA.}
\email{dbath@purdue.edu}
\subjclass[2010]{}
\thanks{This research was supported by a Purdue Research Foundation grant from Purdue University, by NSF Grant 1401392-DMS, and by Simons Foundation Collaboration Grant for Mathematicians \#580839.}

\begin{abstract}

For a central, not necessarily reduced, hyperplane arrangement $f$ equipped with any factorization $f = f_{1} \cdots f_{r}$ and for $f^{\prime}$ dividing $f$, we consider a more general type of Bernstein--Sato ideal consisting of the polynomials $B(S) \in \mathbb{C}[s_{1}, \dots, s_{r}]$ satisfying the functional equation $B(S) f^{\prime} f_{1}^{s_{1}} \cdots f_{r}^{s_{r}} \in \A[s_{1}, \dots, s_{r}] f_{1}^{s_{1} + 1} \cdots f_{r}^{s_{r} + 1}.$

Generalizing techniques due to Maisonobe, we compute the zero locus of the standard Bernstein--Sato ideal in the sense of Budur (i.e. $f^{\prime} = 1)$ for any factorization of a free and reduced $f$ and for certain factorizations of a non-reduced $f$. We also compute the roots of the Bernstein--Sato polynomial for any power of a free and reduced arrangement. If $f$ is tame, we give a combinatorial formula for the roots lying in $[-1,0).$

For $f^{\prime} \neq 1$ and any factorization of a line arrangement, we compute the zero locus of this ideal. For free and reduced arrangements of larger rank, we compute the zero locus provided $\text{deg}(f^{\prime}) \leq 4$ and give good estimates otherwise.
Along the way we generalize a duality formula for $\mathscr{D}_{X,\x}[S]f^{\prime}f_{1}^{s_{1}} \cdots f_{r}^{s_{r}}$ that was first proved by Narv\'aez-Macarro for $f$ reduced, $f^{\prime} = 1$, and $r = 1.$

As an application, we investigate the minimum number of hyperplanes one must add to a tame $f$ so that the resulting arrangement is free. This notion of freeing a divisor has been explicitly studied by Mond and Schulze, albeit not for hyperplane arrangements. We show that small roots of the Bernstein--Sato polynomial of $f$ can force lower bounds for this number.
\end{abstract}

\maketitle

\setcounter{tocdepth}{1}
\tableofcontents

\section{Introduction}

Consider a central, not necessarily reduced, hyperplane arrangement cut out by $f \in \mathbb{C}[X] = \mathbb{C}[x_{1}, \dots, x_{n}].$ Given a factorization $f = f_{1} \cdots f_{r}$, not necessarily into linear terms, and letting $F = (f_{1}, \dots, f_{r})$, there is a free $\mathbb{C}[X][\frac{1}{f}][s_{1}, \dots, s_{r}]$-module generated by the symbol $F^{S} = f_{1}^{s_{1}} \cdots f_{r}^{s_{r}}$. This module has an $\A[S] = \A[s_{1}, \dots, s_{r}]$-module structure, where $\A[S]$ is a polynomial ring extension over the Weyl algebra, given by the formal rules of calculus. We will denote the $\A[S]$-module generated by $F^{S}$ as $\A[S]F^{S}$. For $f^{\prime}$ and $g \in \mathbb{C}[X]$ dividing $f$ we study the polynomials $B(S) \in \mathbb{C}[S] = \mathbb{C}[s_{1}, \dots, s_{r}]$ satisfying the functional equation
\begin{equation} \label{eqn-more general functional}
B(S) f^{\prime} F^{S} \in \A[S] g f^{\prime} F^{S}.
\end{equation}
The ideal populated by said polynomials is the \emph{Bernstein--Sato ideal} $B_{f^{\prime}F}^{g}$. When $f^{\prime} = 1$ and $g = f$ this defines the multivariate Bernstein--Sato ideal in the sense of Budur \cite{BudurLocalSystems} and we simply write $B_{F}$; if we further restrict to the trivial factorization $F = (f)$ then we obtain the classical functional equation whose corresponding ideal, which we denote by $B_{f}$, has as its monic generator the \emph{Bernstein--Sato polynomial}. 

The roots of the Bernstein--Sato polynomial encode various data about the singular locus of $f$. Malgrange and Kashiwara, cf. \cite{Malgrange}, \cite{Kashiwara2}, famously proved that exponentiating the local version of the Bernstein--Sato polynomial's roots recovers the eigenvalues of the algebraic monodromy action on nearby Milnor fibers. In \cite{BudurLocalSystems}, Budur conjectured the analogous claim for the multivariate Bernstein--Sato ideal $B_{F}$ associated to a factorization of $f$ into irreducibles: exponentiating the ideal's zero locus recovers the cohomology support locus of the complement of $\text{Var}(f)$. A proof of this (for germs $f$ that need not be arrangements) has recently been announced by Budur, Veer, Wu, and Zhou, cf. \cite{BudurConjecture}. Beyond these monodromy results, zeroes of Bernstein--Sato polynomials are related to many other invariants: multiplier ideals, log canonical thresholds, F-pure thresholds, etc.

However, even in the case of arrangements, formulae for Bernstein--Sato ideals, polynomials, or their zero loci are very rare.  Walther has found a formula for the Bernstein--Sato polynomial for generic arrangements in \cite{uliGeneric}, Maisonobe has shown the Bernstein--Sato ideal $B_{F}$ for a generic arrangement factored into linear forms is principal and found the corresponding formula for a generator, cf. \cite{MaisonobeGeneric}, and Saito has shown that the roots of the Bernstein--Sato polynomial of a reduced and central arrangement $f$ lie in $(-2 + \frac{1}{\text{deg}(f)},0) \cap \mathbb{Q}$, cf. \cite{SaitoArrangements}. On the other hand, Walther has shown that, in general, the roots of the Bernstein--Sato polynomial are not combinatorially determined, that is, they cannot be computed from the arrangement's intersection lattice, cf. \cite{uli} and Example \ref{ex- roots not combinatorial}. The multivariate Bernstein--Sato ideal $B_{F}$ is not even guaranteed to be principal, cf. \cite{BahloulOaku} for a counter-example in the local case.  To our knowledge, there are no systematic studies of the more general type of Bernstein--Sato ideal $B_{f^{\prime} F}^{g}$ though it does play a role in \cite{uliGeneric}.

Our starting point is the program of Maisonobe in \cite{MaisonobeFree} wherein he proves the Bernstein--Sato ideal of a central, reduced, and free (in the sense of Saito \cite{SaitoLogarithmic}) arrangement equipped with its factorization into linear forms is principal and gives a combinatorial formula for its generator. While the approach is similar, we encounter many technical difficulties because our results are significantly more general: we consider the more general functional equation \eqref{eqn-more general functional} and we often relax the assumptions of $f$ being factored into linear forms, being free, and being reduced. 

\medskip

In Section 2, we consider a larger class of analytic germs $f \in \mathscr{O}_{X}$ than just central, reduced, and free arrangements and we consider any factorization $f = f_{1} \cdots f_{r}$. In \cite{me}, we proved that $\ann_{\D_{X,\x}[S]} F^{S}$ is generated by derivations, that is, by differential operators of order at most one under a natural filtration, under the hypotheses of tameness (a sliding condition on projective dimension), strongly Euler-homogeneous (a hypothesis that a particular logarithmic derivation exists locally everywhere), and Saito-holonomicity (a finiteness condition on the logarithmic stratification). We use similar techniques to generalize these results from \cite{me} in Theorem \ref{thm-gen by derivations}:

\begin{thm}

Suppose $f = f_{1} \cdots f_{r}$ is tame, strongly Euler-homogeneous, and Saito-holonomic, $f^{\prime} \in \bigO_{X,\x}[\frac{1}{f}]$ is compatible with $f$, and $F = (f_{1}, \dots, f_{r}).$ Then the $\D_{X,\x}[S]$-annihilator of $ f^{\prime} F^{S}$ is generated by derivations.

\end{thm}

In Section 3, we replace the hypothesis of tame with free and prove a version of the symmetry of $B_{f^{\prime}F}^{g}$ that was first identified by Narv\'aez-Macarro in \cite{MAcarroDuality} in the case of Bernstein--Sato polynomials and generalized to $B_{F}$ by Maisonobe in \cite{MaisonobeFree}. This follows from computing the $\mathscr{D}_{X,\x}[S]$-dual of $\mathscr{D}_{X,\x}[S] f^{\prime} F^{S}$. Without freeness, computing these $\mathscr{D}_{X,\x}[S]$-duals is currently intractible. While we are certain one could use Narv\'aez-Macarro's Lie-Rinehart strategy, we instead opt for Maisonobe's approach, which itself relies on a computation of the trace of an adjoint action first proved by Castro--Jim\'enez and Ucha in Theorem 4.1.4 of \cite{JimenezUcha}; we give a different proof of this in Appendix A. With $\mathbb{D}$ denoting the $\mathscr{D}_{X,\x}[S]$-dual $\text{RHom}_{\D_{X,\x}[S]}(-, \mathscr{D}_{X,\x}[S])^{\text{left}}$, in Theorem \ref{thm- useful duality formula} we prove:

\begin{thm}
Suppose $f = f_{1} \cdots f_{r} \in \bigO_{X}$ is free, strongly Euler-homogeneous, and Saito-holonomic and $f_{\red} \in \bigO_{X,\x}$ is a Euler-homogeneous reduced defining equation for $f$ at $\x$. Let $F = (f_{1}, \dots, f_{r})$, let $f^{\prime} \in \bigO_{X,\x}$ be compatible with $f$, and let $g \in \bigO_{X,\x}$ such that $f \in \bigO_{X,\x} \cdot g$. Then 
\[
\mathbb{D} \left( \frac{\D_{X,\x}[S]f^{\prime}F^{S}}{\D_{X,\x}[S] \cdot g f^{\prime}F^{S}} \right) \simeq \frac{\D_{X,\x}[S] (g f^{\prime} f_{\red})^{-1} F^{-S}}{\D_{X,\x}[S] (f^{\prime} f_{\red})^{-1} F^{-S}}[n+1].
\]
\end{thm}
The main application is Theorem \ref{thm-unmixed symmetry} which identifies technical conditions on $f^{\prime}, g$, and $F$ such that $B_{f^{\prime}F}^{g}$ is invariant under a non-trivial involution of $\mathbb{C}[S]$.

In Section 4 we return to hyperplane arrangements and first show that the nice structure of $\ann_{\mathscr{D}_{X,\x}{S}} f^{\prime}F^{S}$ from Theorem \ref{thm-gen by derivations} allows us to adapt Maisonobe's arguments to estimate $B_{f^{\prime}F}^{g}$ for any factorization. In particular we complement Walther's result that the roots of Bernstein--Sato polynomial are not combinatorial for even tame arrangements, cf. \cite{uli}. Namely, we prove in Theorem \ref{thm-combinatorial roots} the roots lying in $[-1,0)$ are combinatorial:

\begin{thm}
Let $f$ be a central, not necessarily reduced, tame hyperplane arrangement. Suppose $f^{\prime}$ divides $f$; let $g = \frac{f}{f^{\prime}}.$ Then the roots $\V(B_{f^{\prime}f}^{g})$ lying in $[-1,0)$ are combinatorially determined:
\[
\V(B_{f^{\prime}f}^{g}) \cap [-1,0) = \bigcup_{\substack{ X \in L(A) \\ X \text{ indecomposable}}} \bigcup_{j_{X}= r(X) + d_{X}^{\prime}}^{d_{X}}  \frac{ - j_{X}}{d_{X}}.
\]
Setting $f^{\prime} = 1$ gives all the roots of the Bernstein--Sato polynomial of $f$ lying in $[-1,0).$
\end{thm}

If we assume further that $f$ is free, then we can then use the symmetry property of Theorem \ref{thm-unmixed symmetry} to more accurately estimate $\V(B_{f^{\prime}F}^{g})$, where $\V(-)$ always refers to the zero locus of the ideal in question. In this setting there is a computation for the multivariate Bernstein--Sato ideal of a reduced, free $f$ that has been factored into linear forms due to Maisonobe \cite{MaisonobeFree}, but no results about other factorizations, non-reduced $f$, or even the Bernstein--Sato polynomial.  We fill in much of this gap. With $P_{f^{\prime}F,X}^{g} \in \mathbb{C}[S]$ the explicit linear polynomial from Definition \ref{def- simplifying polynomial}, we obtain the following, which in particular shows that the roots of the Bernstein--Sato polynomial for any power of a reduced, central, and free arrangement are combinatorially determined:

\begin{thm} 
Suppose $f = f_{1} \cdots f_{r}$ is a central, not necessarily reduced, free hyperplane arrangement, $F = (f_{1}, \cdots, f_{r})$, $f^{\prime}$ divides $f$, and $g = \frac{f}{f^{\prime}}.$ If $(f^{\prime}, F)$ is an unmixed pair up to units and if $\deg(f^{\prime}) \leq 4$, then $\V(B_{f^{\prime}F}^{g})$ is a hypersurface and
\begin{equation} \label{eqn-intro variety computation}
\V( B_{f^{\prime}F}^{g}) = \V \left( \prod_{\substack{X \in L(A) \\ X \text{ indecomposable }}} \prod_{j_{X} = 0}^{d_{X,\red} + d_{X} - 2r(X) - d_{X}^{\prime}} \left( P_{f^{\prime}F,X}^{g} + j_{X} \right) \right).
\end{equation}
If $L$ is a factorization of $f = l_{1} \cdots l_{d}$ into irreducibles and $\deg(f^{\prime}) \leq 4$, then
\[
 B_{f^{\prime}L}^{g} =  \prod_{\substack{X \in L(A) \\ X \text{ indecomposable }}} \prod_{j_{X} = 0}^{d_{X,\red} + d_{X} - 2r(X) - d_{X}^{\prime}} \left( P_{f^{\prime}L,X}^{g} + j_{X} \right)
\]
and so $B_{f^{\prime}L}^{g}$ principal. If $f^{\prime} = 1$ and $f$ is reduced, then for any $F$
\begin{equation} \label{eqn - intro b-function root fml}
\V( B_{F}) = \V \left( \prod_{\substack{X \in L(A) \\ X \text{ indecomposable }}} \prod_{j_{X} = 0}^{d_{X,\red} + d_{X} - 2r(X)} \left( P_{F,X}^{g} + j_{X} \right) \right).
\end{equation}
In particular, if $f$ is reduced or is a power of a central, reduced, and free hyperplane arrangement, then the roots of the Bernstein--Sato polynomial of $f$ are given by \eqref{eqn - intro b-function root fml}.
\end{thm}
In Remark \ref{rmk- gesturing at Budur improvement} we discuss how to use new results to get a combinatorial formula for the roots of the Bernstein--Sato polynomial corresponding to any central, free $f$, that is, to $f$ that may not be a power of a reduced arrangement. In the case of line arrangements, we are also able to compute $\V(B_{f^{\prime}F}^{g})$ for any suitable choice of $f^{\prime}, g,$ and $F$ without the technical condition of unmixed up to units, cf. Theorem \ref{thm- rank at most 2 computation} and Definition \ref{def-unmixed}.

Unfortunately our methods are not appropriate for determining the multiplicity of roots of the Bernstein--Sato polynomial so we cannot conclude this polynomial is combinatorial for free arrangements. These multiplicities are mysterious, although in \cite{SaitoArrangements} Saito proves various results about them in the general (i.e. in the non-free) setting. Notably he shows that $-1$ has multiplicity equal to the arrangement's rank. 

In Section 5 we make use of our results involving the more general functional equation \eqref{eqn-more general functional} to study the smallest arrangement $\V(f^{\prime})$ that when added to the arrangement $\V(g)$ makes $\V(f^{\prime}g)$ free, i.e. the smallest arrangement $f^{\prime}$ that \emph{frees} $g.$  For arbitrary divisors $g$, it is unknown whether or not such a divisor $f^{\prime}$ exists. There are some positive results, but the methodologies are very particular to the type of divisors considered. For example, Mond and Schulze identified certain classes of germs that are freed by a adjoint divisors--these germs are related to discrimants of versal deformations, cf. \cite{MondSchulzeAdjoint}. Other cases of freeing divisors are considered in \cite{simis} and \cite{newfromold}. However, Yoshinaga \cite{Yoshi} has communicated to us a way, based on the combinatorics of $g$, to find an arrangement $f^{\prime}$ that frees an arrangement $g$. In Theorem \ref{thm- lower bound on freeing number} we prove the degree of $f^{\prime}$ is related to roots of the Bernstein--Sato polynomial of $g$.
\begin{thm}
Suppose that $g$ is a central, reduced, tame hyperplane arrangement of rank $n$, $v$ an integer such that $1 < v \leq n - 1$, and $\deg(g)$ is co-prime to $v$. If $ \frac{-2 \deg(g) +v}{\deg(g)}$ is a root of the Bernstein--Sato polynomial of $g$ and if $f^{\prime}$ is a central arrangement that frees $g$, then $\deg(f^{\prime}) \geq n - v.$
\end{thm}

In Appendix B we prove a conjecture of Budur's in the case of central, reduced, and free hyperplane arrangments. The recently announced paper \cite{BudurConjecture} gives a general proof using entirely different methods.

\medskip 

We would like to thank Luis Narv\'aez-Macarro, Uli Walther, and Masahiko Yoshinaga for their helpful comments and insights. We would also like to thank the referee for very detailed comments which greatly helped to improve the quality of the text.

\section{Bernstein--Sato Ideals and the $\D_{X,\x}[S]$-module  $\D_{X,\x}[S] f^{\prime} F^{S}$}

In this section we introduce some of our working hypotheses on $f \in \bigO_{X}$. These are needed to utilize results from \cite{me} and \cite{uli} which will be needed throughout the paper. We generalize Theorem 2.29 of \cite{me} and discuss how Bernstein--Sato varieties attached to different factorizations of $f$ relate to each other.

\subsection{Hypotheses on $f$} \text{ }
Let $X$ be a smooth analytic space or $\mathbb{C}$-scheme of dimension $n$ and $\bigO_{X}$ be the analytic structure sheaf. Pick $f \in \bigO_{X}$ to be regular with divisor $Y = \text{Div}(f)$ and ideal sheaf $\mathscr{I}_{Y}$. In general, we make no reducedness assumption on $Y$. 
\begin{define}

Let $\Der_{X}(-\log Y)$ be the $\bigO_{X}$-sheaf of \emph{logarithmic derivations on} $Y$, that is, the sheaf generated locally by the vector fields $\delta$ such that $\delta \bullet \mathscr{I}_{Y} \subseteq \mathscr{I}_{Y}.$ If $Y = \text{Div}(f)$ then we also label $\Der_{X}(- \log f) = \Der_{X}( - \log Y).$ Define the \emph{derivations that kill} $f$ to be
\[
\Der_{X}(- \log_{0} f) = \{ \delta \in \Der_{X}(- \log f) \mid \delta \bullet f = 0 \}. 
\] 
\end{define}

\begin{remark} \label{rmk- log derivations}
\begin{enumerate}[(a)]
    \item It is easily checked that $\Der_{X}(-\log Y)$ depends on $\mathscr{I}_{Y}$ and not the choice of generators of $\mathscr{I}_{Y}$.
    \item By Lemma 3.4 of \cite{FormalStructure}, $\Der_{X,\x}(- \log fg) = \Der_{X,\x}(- \log f) \cap \Der_{X,\x}(- \log g)$. This is not always true when restricting to derivations that kill $f$.
    \item $\Der_{X,\x}(-\log f)$ is closed under taking commutators.
\end{enumerate}
\end{remark}

At points we will be interested in when $\Der_{X}(-\log Y)$ has a particularly nice structure.

\begin{define}

The divisor $Y = \text{Div}(f)$ is \emph{free} when $\Der_{X}(-\log Y)$ is locally everywhere a free $\bigO_{X}$-module. Similarly $f \in \bigO_{X,\x}$ is free when $\Der_{X,\x}(-\log f)$ is a free $\bigO_{X,\x}$-module.

\end{define}

In \cite{SaitoLogarithmic}, Saito introduced the logarithmic differential forms which are, in some sense, a dual notion to logarithmic derivations. 

\begin{define}

Let $\Omega_{X}^{k}$ be the sheaf of differential $k$-forms on $X$ and $d: \Omega_{X}^{k} \to \Omega_{X}^{k+1}$ the standard differential. Define the sheaf of \emph{logarithmic $k$-forms along} $f$ by 
\[
\Omega_{X}^{k}(\log f) = \{ w \in \frac{1}{f} \Omega_{X}^{k} \mid df \wedge w \in \Omega_{X}^{k+1} \}.
\]

An element $f \in \bigO_{X}$ is \emph{tame} if the projective dimension of the logarithmic $k$-forms along $f$ is at most $k$ in each stalk. A divisor $Y$ is tame if it locally everywhere admits tame defining equations. 

\end{define}

\begin{remark} \label{rmk- logarithmic forms remark}
\begin{enumerate}[(a)]
\item The logarithmic 1-forms are dual to the logarithmic differentials: $\text{Hom}_{\bigO_{X,\x}}(\Der_{X,\x}(-\log f), \bigO_{X,\x}) \simeq \Omega_{X}^{1}(\log f).$ When $f$ is free, $\Omega_{X}^{k}(\log f) \simeq \bigwedge^{k} \Omega_{X}^{1}(\log f)$, cf. 1.6 and page 270 of \cite{SaitoLogarithmic}.
\item If $\text{dim}(X) = n \leq 3$ then any divisor $Y$ is automatically tame. This follows from the reflexivity of logarithmic $k$-forms, cf. \cite{SaitoLogarithmic}.
\end{enumerate}
\end{remark}

The logarithmic derivations can also be used to stratify $X$:

\begin{define} (Compare to 3.3 and 3.8 of \cite{SaitoLogarithmic})
There is a relation on $X$ induced by the logarithmic derivations along $Y$. Two points $\x$ and $\y$ are equivalent if there exists an open $U$ containing them and a $\delta \in \Der_{U}(- \log Y \cap U)$ such that: (i) $\delta$ vanishes nowhere on $U$; (ii) an integral curve of $\delta$ passes through $\x$ and $\y$. The transitive closure of this relation stratifies $X$ 
into equivalence classes whose irreducible components are the \emph{logarithmic strata}. These strata constitute the \emph{logarithmic stratification}.

We say $Y$ is \emph{Saito-holonomic} when the logarithmic stratification is locally finite. 
\end{define}

\begin{example} \label{ex- hyperplane saito}
By 3.14 of \cite{SaitoLogarithmic} hyperplane arrangements are Saito-holonomic.
\end{example}

Finally, we define some homogeneity conditions on $f \in \bigO_{X}$.

\begin{define}
We say $f \in \bigO_{X, \x}$ is \emph{Euler-homogeneous} when there exists $\delta \in \Der_{X,\x}(-\log f)$ such that $\delta \bullet f = f.$ If $\delta$ may be picked to vanish at $\x$, then $f$ is \emph{strongly Euler-homogeneous}.

The element $f \in \bigO_{X}$ is (strongly) Euler-homogeneous if it is so at each point. The divisor $Y$ is (strongly) Euler-homogeneous if it locally everywhere admits a defining equation that is (strongly) Euler-homogeneous.
\end{define}

\begin{remark} \label{rmk-strongly Euler}
If $f \in \bigO_{X,\x}$ and $u \in \bigO_{X,\x}$ is a unit, then $f$ is strongly Euler-homogeneous if and only if $u f$ is, cf. Remark 2.8 of \cite{uli}.
\end{remark}

\begin{example} \label{ex-hyperplane strongly Euler}
Hyperplane arrangements are strongly Euler-homogeneous.

\end{example}

Our working hypotheses on $f$ will often be ``tame, strongly Euler-homogeneous, and Saito-holonomic" or ``free, strongly Euler-homogeneous, and Saito-holonomic." In light of Examples \ref{ex- hyperplane saito} and \ref{ex-hyperplane strongly Euler}, if $f$ cuts out a hyperplane arrangement only tameness or freeness need be assumed.

\subsection{The $\D_{X,\x}[S]$-Annihilator of $f^{\prime} F^{S}$} \text{ }

Let $\D_{X}$ be the sheaf of $\mathbb{C}$-linear differential operators with coefficients in $\bigO_{X}$ and $\D_{X}[S]$ be the polynomial ring extension induced by adding $r$ central variables $S = s_{1}, \dots, s_{r}.$

\begin{define} \label{def-Fs def}
Consider the free $\bigO_{X}[S][\frac{1}{f}]$-module generated by the symbol $F^{S} = f_{1}^{s_{1}} \cdots f_{r}^{s_{r}}.$ This is endowed with a $\D_{X}[S]$-action by specifying the action of a $\mathbb{C}$-linear derivation $\delta$ on $\bigO_{X}$. For any $g \in \bigO_{X}[\frac{1}{f}]$, declare
\[
\delta \bullet (s_{i} g F^{S}) = s_{i}(\delta \bullet g) F^{S} + s_{i} g (\sum_{k} \frac{\delta \bullet f_{k}}{f_{k}} s_{k}) F^{S}.
\]
Let $\D_{X}[S] F^{S}$ be the $\D_{X}[S]$-module generated by $F^{S}$. For $g \in \bigO_{X}[\frac{1}{f}]$, let $\D_{X}[S] g F^{S}$ be the $\D_{X}[S]$-module generated by $g F^{S}.$
\end{define}

\begin{remark} \label{rmk- classical fs}
When executing the above construction with only one $s$, we use the notation $\D_{X}[s] f^{s}$. This is the classical, univariate situation.
\end{remark}

In Proposition 2.7 of \cite{me} we showed both that there is a canonical way to associate elements of $\Der_{X}(-\log f)$ to elements of $\ann_{\D_{X}[S]} F^{S}$ and that when $f$ is tame, strongly Euler-homogeneous, and Saito-holonomic, $\ann_{\D_{X,x}[S]} F^{S}$ is generated by said elements. In this subsection we prove the analogous claims for $\ann_{\mathscr{D}_{X,\x}[S]}f^{\prime}F^{S}$, provided $f^{\prime}$ is chosen such that $f^{N} f^{\prime} \in \mathscr{O}_{X,\x}$ and $f^{M} \in \mathscr{O}_{X,\x} \cdot f^{N}f^{\prime}$ for suitable choices of $N, M \geq 0$. First, we show how to associate elements of $\Der_{X,x}(-\log f)$ to $\ann_{\D_{X,x}[S]} f^{\prime} F^{S}$ in an entirely similar way as in the prequel; second, we show that these elements generate $\ann_{\D_{X,x}[S]} f^{\prime} F^{S}$ when $f$ is tame, strongly Euler-homogeneous, and Saito-holonomic.

\begin{define}

The \emph{total order filtration} $F_{(0,1,1)}$ on $\D_{X,\x}[S]$ assigns, in local coordinates, every $\partial_{x_{k}}$ weight one, every $s_{k}$ weight one, and every element of $\bigO_{X}$ weight zero. We will denote the elements of weight at most $l$ by $F_{(0,1,1)}^{l}$ or $F_{(0,1,1)}^{l}(\D_{X,\x}[S]).$ 

\end{define}

\begin{define} \label{def-f prime compatible}
Write $f \in \bigO_{X,\x}$ as $f = u l_{1}^{p_{1}} \cdots l_{q}^{p_{q}}$ where the $l_{t}$ are pairwise distinct irreducibles, $p_{t} \in \mathbb{Z}_{+}$, and $u$ is a unit in $\bigO_{X,\x}$. We say $f^{\prime} \in \bigO_{X,\x}[\frac{1}{f}]$ is  \emph{compatible} with $f$ if there exists a unit $u^{\prime} \in \bigO_{X,\x}$ and integers $v_{t} \in \mathbb{Z}$ such that 
\[
f^{\prime} = u l_{1}^{v_{1}} \cdots l_{q}^{v_{q}}.
\]
In this case, $v_{t}$ is the \emph{multiplicity of $l_{t}$}.
\end{define}

By Remark \ref{rmk- log derivations}, if $f = u l_{1}^{p_{1}} \cdots l_{q}^{v_{q}}$ a factorization of $f$ into irreducibles at $\x$, $u$ a unit, then if $\delta \in \Der_{X,\x}(-\log f), \frac{\delta \bullet l_{t}}{l_{t}} \in \bigO_{X,\x}.$ So for $f^{\prime}$ compatible with $f$,
\[
\delta \bullet f^{\prime}F^{S} = (\delta \bullet f^{\prime}) F^{S} + f^{\prime} (\sum_{k} \frac{\delta \bullet f_{k}}{f_{k}}s_{k}) F^{S} = (\frac{\delta \bullet f^{\prime}}{f^{\prime}} + \sum_{k} \frac{\delta \bullet f_{k}}{f_{k}}s_{k}) f^{\prime}F^{S},
\]
where $(\frac{\delta \bullet f^{\prime}}{f^{\prime}} + \sum_{k} \frac{\delta \bullet f_{k}}{f_{k}}s_{k}) \in \mathscr{O}_{X,\x}[S]$. Indeed, $\frac{\delta \bullet f^{\prime}}{f^{\prime}} = \sum v_{t} \frac{\delta \bullet l_{t}}{l_{t}} \in \mathscr{O}_{X,\x}$ and similarly $\frac{\delta \bullet f_{k}}{f_{k}} \in \mathscr{O}_{X}$.

\begin{define} \label{def- new psi def}
Suppose $f^{\prime}$ is compatible with $f$. If $f = f_{1} \cdots f_{r}$ and $F = (f_{1}, \dots, f_{r})$, then there is a  map of $\bigO_{X,x}$-modules 
\[
\psi_{f^{\prime}F, \x}: \Der_{X,x}(-\log f) \to \ann_{\D_{X,x}[S]}f^{\prime}F^{S} \cap F_{(0,1,1)}^{1}
\]
given by
\[ \psi_{f^{\prime}F, \x}(\delta) = \delta - \sum_{k} \frac{\delta \bullet f_{k}}{f_{k}}s_{k} - \frac{\delta \bullet f^{\prime}}{f^{\prime}}.
\]
The $\bigO_{X,\x}$-module of \emph{annihilating derivations} along $f^{\prime}F$ is defined as
\[
\theta_{f^{\prime}F,\x} = \psi_{f^{\prime}F, \x}(\Der_{X,\x}(-\log f))
\]
and $\ann_{\D_{X,\x}[S]}f^{\prime}F^{S}$ is \emph{generated by derivations} when
\[
\ann_{\D_{X,\x}[S]} f^{\prime}F^{S} = \D_{X,\x}[S] \cdot \theta_{f^{\prime}F,\x}.
\]
When $f^{\prime} = 1$ we write $\psi_{F,\x}$ and $\theta_{F,\x}$.
\end{define}

Arguing as in Proposition 2.7 of \cite{me} we see that:

\begin{prop} \label{prop - f prime psi prop} \text{\normalfont (Compare to Proposition 2.7 of \cite{me})} 
Suppose $f^{\prime}$ is compatible with $f$. If $f = f_{1} \cdots f_{r}$ and $F = (f_{1}, \dots, f_{r})$, then $\psi_{f^{\prime}F, \x}$ is an isomorphism.
\end{prop}

\begin{proof}
Suppose $\delta - \sum_{k} b_{k} s_{k} - b \in  \ann_{\D_{X,\x}[S]} f^{\prime} F \cap F_{(0,1,1)}^{1}$ where $b_{k}, b \in \bigO_{X, \x}.$ Since $ f^{\prime}F^{S}$ generates a free $\bigO_{X,\x}[S][\frac{1}{f}]$-module we deduce 
\[
(\sum_{k} \frac{\delta \bullet f_{k}}{f_{k}}s_{k} - b_{k} s_{k} ) + (\frac{\delta \bullet f^{\prime}}{f^{\prime}} - b) = 0
\]
and hence 
\[
\delta \in \bigcap_{k}\Der_{X, \x}(-\log f_{k}) = \Der_{X,\x}(-\log f).
\]
So the map $\delta - \sum_{k} b_{k} s_{k} - b \mapsto \delta$ sends   $\ann_{\D_{X,\x}[S]} f^{\prime} F \cap F_{(0,1,1)}^{1}$ to $\Der_{X,\x}(-\log f)$. Its inverse is $\psi_{f^{\prime}F,\x}.$
\end{proof}

\begin{remark} \label{rmk-respect commutators}
By definition, $\ann_{\D_{X,\x}[S]}f^{\prime}F^{S}$ is closed under taking commutators; hence $\theta_{f^{\prime}F,\x}$ is as well. As $\psi_{f^{\prime}F,\x}$ is an isomorphism, a basic computation shows $\psi_{f^{\prime}F,\x}$ respects taking commutators.
\end{remark}

In \cite{me} we generalized an approach of Walther's in \cite{uli}: we looked at the associated graded object of $\ann_{\D_{X,\x}[S]} F^{S}$ under the total order filtration $F_{(0,1,1)}$. As $\psi_{F,\x}(\Der_{X,\x}(-\log f)) \subseteq \ann_{\D_{X,\x}[S]} F^{S}$ the following definition is natural:

\begin{define}

Suppose $f$ is strongly Euler-homogeneous. The \emph{generalized Liouville ideal} $\widetilde{L_{F,\x}} \subseteq \gr_{(0,1,1)}(\D_{X,\x}[S])$ is generated by the symbols of elements in $\psi_{F}(\Der_{X,\x}(-\log f))$ under the total order filtration. That is,
\[
\widetilde{L_{F,\x}} = \gr_{(0,1,1)}(\D_{X,\x}[S]) \cdot \gr_{(0,1,1)}(\psi_{F,\x}(\Der_{X,\x}(-\log f))).
\]

\end{define}

\begin{remark} \label{rmk- generalized Liouville ideal}
\begin{enumerate}[(a)]
    \item The strongly Euler-homogeneous assumption in the above definition ensures that algebraic properties of $\widetilde{L_{F,x}}$ do not depend on choice of defining equations for each $f_{k}$ at $x$. See Remark 2.15 of \cite{me} for details.
    \item By Corollary 2.28 of \cite{me}, if $f \in \bigO_{X}$ is tame, strongly Euler-homogeneous, and Saito-holonomic then $\widetilde{L_{F,\x}} = \gr_{(0,1,1)}(\ann_{\D_{X,\x}[S]} F^{S}).$ 
    \item For $\delta \in \Der_{X,\x}(-\log f)$, note that 
    \begin{align*}
    \gr_{(0,1,1)}(\psi_{f^{\prime} F, \x}(\delta)) 
        &= \gr_{(0,1,1)}(\delta - \sum_{k} \frac{\delta \bullet f_{k}}{f_{k}} s_{k} - \frac{\delta \bullet f^{\prime}}{f^{\prime}}) \\
        & =  \gr_{(0,1,1)}(\delta - \sum_{k} \frac{\delta \bullet f_{k}}{f_{k}} s_{k}) \\
        &= \gr_{(0,1,1)}(\psi_{F,\x}(\delta)).
    \end{align*}
    Since $\widetilde{L_{F,\x}} \subseteq \gr_{(0,1,1)}(\mathscr{D}_{X,\x}[S])$ has, by definition, generators $\{ \gr_{(0,1,1)}(\psi_{F,\x}(\delta)) \mid \delta \in \Der_{X,\x}(-\log f) \}$, we deduce 
    \begin{align*}
        \widetilde{L_{F,\x}} 
            &= \gr_{(0,1,1)}(\mathscr{D}_{X,\x}[S]) \cdot \{ \gr_{(0,1,1)}(\psi_{f^{\prime}F,\x}(\delta)) \mid \delta \in \Der_{X,\x}(-\log f) \} \\
            &= \gr_{(0,1,1)}(\D_{X,\x}[S]) \cdot \gr_{(0,1,1)}(\theta_{f^{\prime}F,\x}) \\
            & \subseteq \gr_{(0,1,1)}(\ann_{\mathscr{D}_{X,\x}[S]} f^{\prime}F^{S}).
    \end{align*}

\end{enumerate}
\end{remark}

By the preceding remark, $\widetilde{L_{F,\x}}$ approximates $\gr_{(0,1,1)}(\ann_{\D_{X,\x}[S]} f^{\prime} F^{S}).$ Arguing as in Corollary 2.28 of \cite{me} we prove the approximation is in fact an equality:

\begin{thm} \label{thm-gr gen by derivations}
Suppose $f = f_{1} \cdots f_{r}$ is tame, strongly Euler-homogeneous and Saito-holonomic. Let $F = (f_{1}, \dots, f_{r})$ and suppose $f^{\prime} \in \bigO_{X,\x}[\frac{1}{f}]$ is compatible with $f$. Then
\[\gr_{(0,1,1)}(\ann_{\D_{X,\x}[S]} f^{\prime} F^{S}) = \gr_{(0,1,1)}(\D_{X,\x}[S]) \cdot\gr_{(0,1,1)}(\theta_{f^{\prime}F,\x}).
\]
\end{thm}

\begin{proof}
For the first part of this proof we mimic Proposition 2.25 of \cite{me}. In Definition 2.24 of loc. cit. we introduced a $\bigO_{X,\x}$-linear ring homomorphism $\phi_{F,\x} : \gr_{(0,1,1)}(\D_{X,\x}[S]) \to R (\text{Jac}(f_{1}), \dots, \text{Jac}(f_{r}))$ where $R (\text{Jac}(f_{1}), \dots, \text{Jac}(f_{r}))$ is the multi-Rees algebra associated to the $r$ Jacobian ideals $\text{Jac}(f_{1}), \dots, \text{Jac}(f_{r})$. Using local coordinates $\partial_{x_{i}}$ and identifying $\gr_{(0,1,1)}(\D_{X,\x}[S])$ with $\bigO_{X,\x}[Y][S]$ via  $\gr_{(0,1,1)}(\partial_{x_{i}}) = y_{i}$, the map $\phi_{F,\x}$ is given by
\[
y_{i} \mapsto \sum_{k} \frac{f}{f_{k}} (\partial_{x_{i}} \bullet f_{k}) s_{k} \text{ and } s_{k} \mapsto f s_{k}.
\]
Proposition 2.26 of loc. cit. shows $\text{ker}(\phi_{F,\x})$ is a prime ideal of dimension $n + r.$

Select $P \in \ann_{\D_{X,\x}[S]} f^{\prime} F$ of weight $l$ under the total order filtration $F_{(0,1,1)}$. For any $Q$ of weight $l$, $f^{l} Q \bullet f^{\prime} F^{S} \in \bigO_{X,\x}[S] F^{S}$. Now, for $g \in \mathscr{O}_{X,\x}[S][\frac{1}{f}]$, write $\partial_{x_{i}} \bullet gf^{\prime}F^{S} = (\partial_{x_{i}} \bullet g + g \frac{\partial_{x_{i}} \bullet f^{\prime}}{f^{\prime}} + g \sum_{k} \frac{\partial_{x_{i}} \bullet f_{k}}{f_{k}}s_{k}) f^{\prime} F^{S}$. Thus, if applying a partial derivative to $g f^{\prime} F^{S}$ causes the $s$-degree (under the natural filtration) of the $\mathscr{O}_{X,\x}[S]$-coefficient of $f^{\prime}F^{S}$ to increase, the terms of higher $s$-degree are precisely $g \sum_{k} \frac{\partial_{x_{i}} \bullet f_{k}}{f_{k}}$. A straightforward computation then shows that the $S$-lead term of $ f^{l} Q \bullet f^{\prime} F^{S}$ is exactly $\phi_{F,\x}(\gr_{(0,1,1)}(Q)) f^{\prime} F^{S} \in \bigO_{X,\x}[S] f^{\prime}F^{S}.$ Since $f^{\prime} F^{S}$ generates a free $\bigO_{X,\x}[S][\frac{1}{f}]$-module and since $P \bullet f^{\prime} F^{S} = 0$, we conclude $\gr_{(0,1,1)}(P) \in \text{ker}(\phi_{F,\x}).$

By Remark \ref{rmk- generalized Liouville ideal} we deduce:
\begin{align} \label{eqn- containment chain}
    \widetilde{L_{F,\x}}  \subseteq \gr_{(0,1,1)}(\D_{X,\x}[S]) \cdot \gr_{(0,1,1)}(\theta_{f^{\prime}F,\x}) 
        &\subseteq \gr_{(0,1,1)}(\ann_{\D_{X,\x}[S]} f^{\prime}F^{S}) \\
        &\subseteq \text{ker}(\phi_{F,\x}). \nonumber
\end{align}
Since $f$ is tame, strongly Euler-homogeneous, and Saito-holonomic, by Theorem 2.23 of loc. cit., $\widetilde{L_{F,\x}}$ is a prime ideal of dimension $n+r$. So the outer ideals of \eqref{eqn- containment chain} are prime ideals of dimension $n+r$ and the containments are equalities.
\end{proof}

\begin{thm} \label{thm-gen by derivations}

Suppose $f = f_{1} \cdots f_{r}$ is tame, strongly Euler-homogeneous, and Saito-holonomic, $f^{\prime} \in \bigO_{X,\x}[\frac{1}{f}]$ is compatible with $f$, and $F = (f_{1}, \dots, f_{r}).$ Then the $\D_{X,\x}[S]$-annihilator of $ f^{\prime} F^{S}$ is generated by derivations.

\end{thm}

\begin{proof}
By Theorem \ref{thm-gr gen by derivations}, for $P \in \ann_{\D_{X,\x}[S]} f^{\prime}F^{S}$, we can find $L \in \mathscr{D}_{X,\x}[S] \cdot \theta_{f^{\prime}F,\x}$ such that $P$ and $L$ have the same initial term with respect to the total order filtration. Since $P - L$ annihilates $f^{\prime}F^{S}$ and, by construction, has a smaller weight than $P$, we can argue inductively as in Theorem 2.29 of \cite{me} now using Theorem \ref{thm-gr gen by derivations} instead of Corollary 2.28 of \cite{me}. The induction argument therein will also terminate in this setting since  $\ann_{\D_{X,\x}[S]}f^{\prime}F^{S} \cap \bigO_{X,\x} = 0.$
\end{proof}

The following corollary will let us study the Weyl algebra version of the annihilator of $f^{\prime}F^{S}$ when $f^{\prime}$ and $f$ are global algebraic.

\begin{cor} \label{cor-algebraic category}
If $X$ is the analytic space of a smooth $\mathbb{C}$-scheme, then the statement of Theorem \ref{thm-gen by derivations} holds in the algebraic category.
\end{cor}
\begin{proof} 
See Corollary 2.30 of \cite{me}.
\end{proof}

We will also be interested in the $\D_{X,\x}[S]$-module generated by the symbol $F^{-S} = f_{1}^{-s_{1}} \cdots f_{r}^{-s_{r}}$ which is defined in the same way as $\D_{X,\x}[S]F^{S}$. Most of our previous definitions apply to $F^{-S}$ as well, in particular, if $f^{\prime}$ is compatible with $f$ let $\psi_{f^{\prime}F, \x}^{-S}$ and $\theta_{F,\x}^{-S}$ be as before, except with the signs of the $s_{k}$ switched.

\begin{thm} \label{thm-gen by derivations -s}
Suppose $f = f_{1} \cdots f_{r}$ is tame, strongly Euler-homogeneous, and Saito-holonomic, $f^{\prime}$ is compatible with $f$, and $F = (f_{1}, \cdots, f_{r}).$ Then the $\D_{X,\x}[S]$-annihilator of $f^{\prime}F^{-S}$ is generated by derivations in that 
\[
\ann_{\D_{X,\x}[S]}f^{\prime}F^{-S} = \D_{X,\x}[S] \cdot \theta_{f^{\prime}F, \x}^{-S}. 
\]
If $X$ is the analytic space of a smooth $\mathbb{C}$-scheme, then this holds in the algebraic category as well.
\end{thm}

\begin{proof}
It is sufficient to prove the generated by derivations statement. For this argue as in Theorem \ref{thm-gen by derivations} except replace $\widetilde{L_{F,\x}}$ and $\phi_{F,\x}$ with their images under the $\gr_{(0,1,1)}(\D_{X,\x}[S])$ automorphism induced by $s_{k} \mapsto -s_{k}.$
\end{proof}

\subsection{Bernstein--Sato Ideals} \label{subsection- BS varieties} \text{ }

Recall the univariate \emph{functional equation}, with $b(s) \in \mathbb{C}[s]$, $P(s) \in \D_{X,\x}[s]$:
\[
b(s) f^{s} = P(s) f^{s+1}.
\]
The polynomials $b(s)$ generate the \emph{Bernstein--Sato ideal} $B_{f,\x}$ of $f$. The monic generator of this ideal is the \emph{Bernstein--Sato polynomial}; the reduced locus of its variety is $\V(B_{f,\x})$. We will be interested in multivariate generalizations of this functional equation. 

\begin{define} \label{def-multivariate BS ideals generalization}
Let $f^{\prime}, g_{1}, \dots, g_{u} \in \bigO_{X,\x}$ and $I$ the ideal generated by the $g_{1}, \dots g_{u}$. Consider the functional equation
\[
B(S)f^{\prime}F^{S} = \sum\limits_{t} P_{t} g_{t} f^{\prime} F^{S} \in \D_{X,\x}[S] \cdot I f^{\prime} F^{S}
\]
where $f =  f_{1} \cdots f_{r}$, $F = (f_{1}, \dots, f_{r})$, $P_{t} \in \D_{X,\x}[S]$, and $B(S) \in \mathbb{C}[S].$ The polynomials $B(S)$ satisfying this functional equation constitute the \emph{Bernstein--Sato ideal} $B_{f^{\prime}F,\x}^{I}$. Note that 
\[
B_{f^{\prime}F,\x}^{I} = \mathbb{C}[S] \cap ( \ann_{\D_{X,\x}[S]}f^{\prime} F + \D_{X,\x}[S] \cdot I).
\]

When $I = (f)$ we will write $B_{f^{\prime}F,\x}^{I} = B_{f^{\prime}F,\x}$ and when $I = (g)$ we will write $B_{f^{\prime}F,\x}^{g}.$ When in the univariate case, i.e. $r=1$, we will write $B_{f^{\prime}F,\x} = B_{f^{\prime}f,x}$ and $B_{f^{\prime}F,\x}^{g} = B_{f^{\prime}f,\x}^{g}$. When in the global algebraic case we define similar objects using $\A[S]$ instead of $\D_{X,\x}[S]$--in this case we drop the $(-)_{\x}$ subscript. Finally by $\V(-)$ we always mean the reduced locus of the appropriate variety.
\end{define}

We will want to compare the Bernstein--Sato ideals corresponding to different factorizations.

\begin{define} \label{def-compatible with G}
Let $f = f_{1} \cdots f_{r}$ and $F = (f_{1}, \dots, f_{r})$. Write $[r]$ as the disjoint union of the intervals $I_{t}$ where $1 \leq t \leq m$ and consider the \emph{coarser} factorization $H = (h_{1}, \dots, h_{m})$ where $f = h_{1} \cdots h_{m}$ and $h_{t} = \prod_{i \in I_{t}} f_{i}$. Define $S_{H}$ to be the ideal of $\mathbb{C}[S]$ generated by $s_{i} - s_{j}$ for all $i, j \in I_{t}$ and for all $t$. 
\end{define}

\begin{prop} \label{prop-BS ideal subset, coarser}
Let $f = f_{1} \cdots f_{r}$ be tame, strongly Euler-homogeneous, and Saito-holonomic. Let $F = (f_{1}, \dots, f_{r})$, let $I \subseteq \bigO_{X,\x}$, and let $H$ be a coarser factorization. If $f^{\prime} \in \bigO_{X,\x}$ such that $f \in \bigO_{X,\x} \cdot f^{\prime}$, then the image of $B_{f^{\prime}F, \x}^{I}$ modulo $S_{H}$ lies in $B_{f^{\prime}H, \x}^{I}$.
\end{prop}
\begin{proof}
As $f^{\prime}$ is compatible with $f$, $\ann_{\D_{X,\x}[S]}f^{\prime}F^{S}$ and $\ann_{\D_{X,\x}[S]}f^{\prime}H^{S}$ are both generated by derivations. Since $\Der_{X,\x}(-\log f) \subseteq \Der_{X,\x}(-\log f^{\prime})$, we can easily get a result similar to Proposition 2.33 of \cite{me} and, from that, a result similar to Proposition 2.32 of loc. cit. The argument is essentially the same as the proof of Proposition \ref{prop- freeing arrangements, embedding} of this paper.
\end{proof}

\begin{example} \label{ex- going modulo}
For $f = xy^2(x+y)^2$ and $F = (xy, y(x+y), x+y)$, 
\[
B_{F} = (s_{1} + 1) \prod_{j=0}^{1} (s_{1} + s_{2} + 1 + j)(s_{2} + s_{3} + 1 +j)(\prod_{m=0}^{4} (2s_{1} + 2s_{2} + s_{3} + 2 + m).
\]
While Proposition \ref{prop-BS ideal subset, coarser} can estimate $B_{f}$, it estimates multiplicities poorly. Indeed, going modulo $(s_{1}-s_{2}, s_{1}-s_{3}, s_{2} - s_{3})$ we find 
\[
(s + 1)^{3} (2s + 1)^{2}\prod_{m=0}^{4} (5s + 2 + m) \in B_{f} = \mathbb{C}[s] \cdot (s+1) (2s+ 1) \prod_{m=0}^{4} (5s + 2 + m).
\]
\end{example}

\section{$\D_{X,\x}[S]$-Dual of $\D_{X,\x}[S]f^{\prime}F^{S}$}

In \cite{MAcarroDuality}, Narv\'aez-Macarro computed the $\D_{X,\x}[s]$-dual of $\D_{X,\x}[s]f^{s}$ when $f$ is reduced, free, and quasi-homogeneous; in \cite{MaisonobeFree} Maisonbe generalized this approach to compute the $\D_{X,\x}[S]$-dual of $\D_{X,\x}[S]F^{S}$ where $f$ is as in \cite{MAcarroDuality}, $f = f_{1} \cdots f_{r}$, and $F = (f_{1}, \dots, f_{r}).$ In this section we will use Maisonobe's approach to compute the $\D_{X,x}[S]$-dual of $\D_{X,x}[S]f^{\prime}F^{S}$ where $f \in \bigO_{X}$ is free, strongly Euler-homogeneous, Saito-holonomic, not necessarily reduced but admitting a reduced Euler-homogeneous defining equation $f_{\red}$ at $\x$, $f^{\prime} \in \bigO_{X,\x}$ is compatible with $f$, and $F = (f_{1}, \dots, f_{r})$ corresponds to any factorization, not necessarily into irreducibles, of $f = f_{1} \cdots f_{r}$. The strategy hinges on a formula for the trace of the adjoint first proved by Castro--Jim\'enez and Ucha in Theorem 4.1.4 of \cite{JimenezUcha}. We supply a different proof in Proposition \ref{prop: main adjoint right-left formula}.  

In the second subsection, we note that this duality computation lets us argue as in Maisonobe's Proposition 20 of \cite{Maisonobetheory} and prove that the radical of $B_{f^{\prime}F,\x}$ is principal. In the third subsection, we show that $B_{f^{\prime}F,\x}^{g}$ is fixed under a non-trivial involution when $f^{\prime}$, $F$, and $g$ satisfy a technical condition, cf. Definition \ref{def-unmixed}. 

\begin{convention}
A resolution is a (co)-complex with a unique (co)-homology module at its end. An acyclic (co)-complex has no non-trivial (co)-homology. Given a (co)-complex ($C^{\bullet}$) $C_{\bullet}$ resolving $A$, the augmented (co)-complex ($C^{\bullet} \to A$) $C_{\bullet} \to A$ is acyclic. 
\end{convention}

\subsection{Computing the Dual} \text{ }

Our argument begins at essentially the same place as Narv\'aez-Macarro's and Maisonobe's: the Spencer co-complex.

\begin{define} \label{def-extended Spencer}
Let $f = f_{1} \cdots f_{r} \in \bigO_{X,\x}$ be free, let $F = (f_{1}, \dots, f_{r})$, and let $f^{\prime} \in \bigO_{X,\x}$ be compatible with $f$. Consider $g_{1}, \dots, g_{u} \in \mathscr{O}_{X,\x}$ such that $f \in \mathscr{O}_{X,\x} \cdot g_{j}$ for all $1 \leq j \leq u$, and let $I \subseteq \mathscr{O}_{X,\x}$ be the ideal generated by $g_{1}, \dots, g_{u}$. We will define $\Sp_{\theta_{f^{\prime}F,\x}}^{I}$, the \emph{extended Spencer co-complex} associated to $f^{\prime}$ and $I$. When $I = (g)$, write $\Sp_{f^{\prime}F}^{g}$. This will be a mild generalization of the normal Spencer complex, cf. A.18 of \cite{MAcarroDuality}.

Let $E$ be the free submodule of $\bigO_{X,\x}^{u}$ prescribed by the basis $e_{1}, \dots, e_{u}$ where $e_{j} = (0, \dots, g_{j}, \dots, 0).$ We define an anti-commutative map
\[
\sigma: (\theta_{f^{\prime}F,\x} \oplus E )\times ( \theta_{f^{\prime}F,\x} \oplus E) \to \theta_{f^{\prime}F,\x} \oplus E
\] 
that is essentially the commutator on $F_{(0,1,1)}^{1}(\D_{X,\x}[S]).$ The map is determined by its anti-commutativity and the following assignments:
\begin{align*}
    \sigma(\lambda_{i}, \lambda_{j})
        = \left\{ 
            \begin{array}{cc}
            [\lambda_{i},\lambda_{j}],   & \lambda_{i}, \lambda_{j} \in \theta_{f^{\prime}F,\x}, \\
            0,       & \lambda_{i}, \lambda_{j} \in E, \\
            \frac{\delta \bullet (b g_{j})}{g_{j}} e_{j}, & \lambda_{i} = \psi_{f^{\prime}F, \x}(\delta_{i}) \text{ for } \delta_{i} \in \Der_{X,\x}(-\log f), \ \lambda_{j} = b e_{j}.
            \end{array}
            \right.
\end{align*}
Abbreviate $\Sp_{\theta_{f^{\prime}F,\x}}^{I}$ as $\Sp^{\bullet}.$ Then the objects of our complex are
\[
\Sp^{-m} = \D_{X,\x}[S] \otimes_{\bigO_{X,\x}} \bigwedge^{m} ( \theta_{f^{\prime}F,\x} \oplus E )
\]
and the differentials $d^{-m} : \text{Sp}^{-m} \mapsto \text{Sp}^{-m+1}$ are given by 
\begin{align*}
d^{-m} ( P \otimes \lambda_{1} \wedge \cdots \wedge \lambda_{m} ) 
    &= \sum_{i=1}^{r} (-1)^{i-1} P \lambda_{i} \otimes \widehat{\lambda_{i}} \\
    &+ \sum\limits_{1 \leq i < j \leq m} (-1)^{i + j} P \otimes \sigma(\lambda_{i}, \lambda_{j}) \wedge \widehat{\lambda_{i,j}}.
\end{align*}
Here $\widehat{\lambda_{i}}$ is the wedge, in increasing order, of all the $\lambda_{1}, \dots ,\lambda_{r}$ except for $\lambda_{i}$; $\widehat{\lambda_{i,j}}$ is the same except now excluding both $\lambda_{i}$ and $\lambda_{j}$. To be clear, we interpret $P e_{j}$ as $Pg_{j} \in \mathscr{D}_{X,\x}[S]$; in particular, $d^{-1}(P \otimes e_{j}) = Pg_{j}.$ There is a natural augmentation map 
\[
\text{Sp}^{0} = \D_{X,\x}[S] \mapsto \frac{\D_{X,\x}[S]}{\D_{X,\x}[S] \cdot \theta_{f^{\prime}F,\x} + \D_{X,\x}[S] \cdot I} 
.
\]

\end{define}

\begin{remark} \label{rmk-spencer remark}
\begin{enumerate}[(a)]
    \item Since $\Der_{X,\x}(-\log f)$ is closed under taking commutators, so is $\theta_{f^{\prime}F,\x}$, see also Example 4.7 of \cite{me}. And as $g_{j}$ divides $f$ for all $1 \leq j \leq u,$ we know $\Der_{X,\x}(-\log f) \subseteq \Der_{X,\x}(-\log g_{j})$ for all $j$. Thus $\sigma$, and consequently the differentials, are well-defined. 
    \item That the extended Spencer co-complex is in fact a co-complex is a straightforward computation mirroring the case of the standard Spencer co-complex.
    \item We have assumed $f$ is free so that $\Sp_{\theta_{f^{\prime}F,\x}}^{I}$ will be a finite, free co-complex of $\D_{X,\x}[S]$-modules. We may fix a basis of $\theta_{f^{\prime}F,\x}$, extend it to a basis of $\theta_{f^{\prime}F,\x} \oplus E$ using the prescribed basis of $E$, and then compute differentials. Label this basis $\lambda_{1}, \dots, \lambda_{n+u}$. Let $\sigma(\lambda_{i}, \lambda_{j}) = \sum_{k =1}^{n+u} c_{k}^{i, j} \lambda_{k}$ be the unique expression of $\sigma(\lambda_{i}, \lambda_{j})$. Then 
    \begin{align*}
    d^{-m} (\lambda_{1} \wedge \cdots \wedge \lambda_{m}) 
        &= \sum_{i=1}^{m} (-1)^{i-1}\lambda_{i} \otimes \widehat{\lambda_{i}} \\
        &+ \sum_{1 \leq i < j \leq m} (-1)^{i+j} c_{i}^{i,j} \otimes (-1)^{i-1} \widehat{\lambda_{j}} + (-1)^{i+j} c_{j}^{i,j} \otimes (-1)^{j} \widehat{\lambda_{i}} \\
        &=  \sum_{i=1}^{m} \left( (-1)^{i-1} \lambda_{i} + \sum_{ j < i}(-1)^{i-1}c_{j}^{j,i} + \sum_{ i < j} (-1)^{i} c_{j}^{i,j} \right) \otimes \widehat{\lambda_{i}}.
    \end{align*}
    We can naturally encode this as matrix multiplication on the right.
\end{enumerate}
\end{remark}

The following calculation relies on Castro--Jim\'enez and Ucha's formula for adjoints appearing in Theorem 4.1.4 of \cite{JimenezUcha}; cf. Proposition \ref{prop: main adjoint right-left formula} for our proof. See also Lemma 1 and Proposition 6 of \cite{MaisonobeFree}. Before stating the Proposition, let us recall the side-changing functor for $\mathscr{D}_{X,\x}[S]$-modules. We use the notation of Appendix A of \cite{MAcarroDuality}.

\begin{define} \text{\normalfont (Compare to Appendix A of \cite{MAcarroDuality})} \label{def-side changing functor}
We will define the equivalence of categories between right $\mathscr{D}_{X,\x}[S]$-modules and left $\mathscr{D}_{X,\x}[S]$-modules. First, regard $\Der_{X,\x}[S]$ as a free $\mathscr{O}_{X,\x}[S]$-module of rank $n$. Then the dualizing module $\omega_{\Der_{X,\x}[S]}$ of $\Der_{X,\x}[S]$ is 
defined as 
\[
\omega_{\Der_{X,\x}[S]} = \Hom_{\mathscr{O}_{X,\x}[S]} \left( \bigwedge^{n} \Der_{X,\x}[S], \mathscr{O}_{X,\x}[S] \right).
\]
This naturally carries a right $\mathscr{D}_{X,\x}[S]$-module structure by A.20 of \cite{MAcarroDuality}. The aforementioned equivalence of categories is given by associated to every right $\mathscr{D}_{X,\x}[S]$-module $Q$ the left $\mathscr{D}_{X,\x}[S]$-module $Q^{\text{left}}$ defined by 
\[
Q^{\text{left}} = \Hom_{\mathscr{O}_{X,\x}[S]} \left( \omega_{\Der_{X,\x}[S]}, Q \right).
\]
That $Q^{\text{left}}$ is a left $\mathscr{D}_{X,\x}[S]$-module follows from A.2 of \cite{MAcarroDuality}; that this gives an equivalence of categories follows from the discussion before A.25 of loc. cit.
\end{define}

\begin{remark} \label{rmk-side changing tau}
Despite the $s$-terms, this side-changing functor is defined entirely similarly to the side-changing functor for $\mathscr{D}_{X,\x}$-modules. So just as in the $\mathscr{D}_{X,\x}[S]$-module case, if we fix coordinates $(x, \partial_{x})$ we can describe the transition from right to left $\mathscr{D}_{X,\x}[S]$-modules in elementary terms. Define $\tau : \mathscr{D}_{X,\x}[S] \to \mathscr{D}_{X,\x}[S]$ by $\tau ( x^{\alpha} \partial_{x}^{\beta} s^{\gamma}) = (- \partial_{x}^{\beta}) x^{\alpha} s^{\gamma}$ where $\alpha, \beta,$ and $\gamma$ are multi-indices. Then $(-)^{\text{left}}$ sends the cyclic right $\mathscr{D}_{X,\x}[S]$-module $\mathscr{D}_{X,\x}[S] / J$ is to the left $\mathscr{D}_{X,\x}[S]$-module $\mathscr{D}_{X,\x}[S] / \tau(J).$ See 1.2 of \cite{uliComputingHom} for details in a similar case.
\end{remark}

\begin{prop} \label{prop-dual spencer terminal homology}
Let $f = f_{1} \cdots f_{r} \in \bigO_{X,\x}$ be free, $F = (f_{1}, \dots, f_{r})$, $f_{\red} \in \bigO_{X,\x}$ a Euler-homogeneous reduced defining equation for $f$ at $\x$, and $I \subseteq \bigO_{X,\x}$ the ideal generated by $g_{1}, \dots, g_{u}$ with $f \in \bigO_{X,\x} \cdot g_{v}$ for each $g_{v}$. Write $g = g_{1} \cdots g_{u}$. Then we can compute the terminal homology module of $\Hom_{\D_{X,\x}[S]} ( \Sp_{\theta_{f^{\prime}F,x}}^{I}, \D_{X,\x}[S])^{\text{\normalfont left}}$:
\[
H_{-n-u} \left( \Hom_{\D_{X,\x}[S]} ( \Sp_{\theta_{f^{\prime}F,x}}^{I}, \D_{X,\x}[S])^{\text{\normalfont left}} \right) \simeq \frac{\D_{X,\x}[S]}{\D_{X,\x}[S] \cdot \theta_{(f^{\prime}gf_{\red})^{-1}F, \x}^{-S} + \D_{X,\x}[S] \cdot I}.
\]
\end{prop}

\begin{proof}

We will show that the image of $\Hom_{\D_{X,\x}[S]}(d^{-n-u}, \D_{X,\x}[S])^{\text{left}}$ is $\D_{X,\x}[S] \cdot \theta_{(f^{\prime}gf_{\red})^{-1}F, \x}^{-S} + \D_{X,\x}[S] \cdot I.$ It suffices to do this in local coordinates $x_{1}, \dots, x_{n}$. Select a basis $\delta_{1}, \dots, \delta_{n}$ of $\Der_{X,\x}(-\log f)$, label $\lambda_{i} = \psi_{f^{\prime}F, \x}(\delta_{i})$ and label $\lambda_{n+j} = e_{j} = (0, \dots, g_{j}, \dots, 0)$ for $1 \leq j \leq u$, cf. Definition \ref{def-extended Spencer}. Then $\lambda_{1}, \dots, \lambda_{n+u}$ is a basis of $\psi_{f^{\prime}F,\x} \oplus E$. Consequently, we may uniquely write $\sigma(\lambda_{i}, \lambda_{j}) = \sum_{k =1}^{n+u} c_{k}^{i, j} \lambda_{k}$ with $c_{k}^{i,j} \in \mathscr{O}_{X,\x}$. 

Let us compute the $c_{k}^{i,j}$ terms in cases. First assume $i, j \leq n$. Then $\sigma(\lambda_{i}, \lambda_{j}) = [ \psi_{f^{\prime}F,\x}(\delta_{i}), \psi_{f^{\prime}F,\x}(\delta_{j})] = [\delta_{i}, \delta_{j}]$, where the last equality follows since $\psi_{f^{\prime}F,\x}$ respects taking commutators, cf. Remark \ref{rmk-respect commutators}. Thus $c_{1}^{i,j}, \dots, c_{n}^{i,j}$ satisfy $[ \delta_{i}, \delta_{j}] = \sum_{k=1}^{n} c_{k}^{i,j} \delta_{k}$; moreover, if $k \geq n+1$, then $c_{k}^{i,j} = 0.$ Second, assume $i \leq n$ and $j \leq u$. By definition $\sigma(\lambda_{i}, \lambda_{n+j}) = \frac{\delta \bullet g_{j}}{g_{j}}\lambda_{n+j}$ and so $c_{n+j}^{i,n+j} = \frac{\delta_{i} \bullet g_{j}}{g_{j}}$ and $c_{k}^{i,n+j} = 0$ for $k \neq n+j$. Similarly for $j \leq n$ and $i \leq u$, $c_{n+j}^{n+j,i} = - \frac{\partial_{i} \bullet g_{j}}{g_{j}}$ and $c_{k}^{n+j,i} = 0$ for all $k \neq n+j$. Finally, assume $i, j \leq u$. Then $\sigma(\lambda_{n+i}, \lambda_{n+j}) = 0$ and $c_{k}^{n+i, n+j} = 0$ for all $k$.

Using Remark \ref{rmk-spencer remark}, $d^{-n-u}$ is given, where $i \leq n$ and $v \leq u$, by multiplying on the right by the matrix
\begin{equation} \label{matrix- d n plus u}
\begin{bmatrix}  
\cdots & (-1)^{i-1}( \psi_{f^{\prime}F,\x}(\delta_{i}) - \sum\limits_{j=1}^{n} c_{j}^{i, j} - \sum\limits_{v=1}^{u} \frac{\delta \bullet g_{v}}{g_{v}}) & \cdots & (-1)^{n+v-1} g_{v} & \cdots 
\end{bmatrix}.
\end{equation}
The dual map is given by transposing $\eqref{matrix- d n plus u}$ and applying $\tau$, the standard right-to-left map (cf. Remark \ref{rmk-side changing tau}), to each each entry where $\tau$ is inert on $\bigO_{X,\x}[S]$ and sends $h \partial_{x_{i}}$ to $- \partial_{x_{i}} h$, $h \in \bigO_{X,\x}[S].$ Write $\delta_{i} = \sum_{e} h_{e,i} \partial_{x_{e}}$ and observe that $\tau(\delta_{i}) = - \delta_{i} - \sum_{e} \partial_{x_{e}} \bullet h_{e,i}$. Therefore $\Hom_{\D_{X,\x}[S]}(d^{-n-u}, \D_{X,\x}[S])^{\text{left}}$ is given by right multiplication by
\begin{equation} \label{matrix- d n plus u dual with all terms}
\begin{bmatrix}
\vdots \\
(-1)^{i-1}( - \delta_{i} - \sum\limits_{k=1}^{r} \frac{\delta_{i} \bullet f_{k}}{f_{k}}s_{k} - \frac{\delta_{i} \bullet f^{\prime}}{f^{\prime}} - \sum\limits_{e=1}^{n} \partial_{x_{e}} \bullet h_{e,i} - \sum\limits_{j=1}^{n} c_{j}^{i, j} - \sum\limits_{v=1}^{u} \frac{\delta \bullet g_{v}}{g_{v}}) \\
\vdots \\
(-1)^{n+v-1}g_{v} \\
\vdots \\
\end{bmatrix}
\end{equation}

Assume $n \geq 2$. We could have chosen $\delta_{1}, \dots, \delta_{n}$ to be a preferred basis of $\Der_{X,\x}(-\log f_{\red}) = \Der_{X,\x}(-\log f)$, cf. Definition \ref{def-preferred basis}, making $\delta_{1}, \dots, \delta_{n-1} \in \Der_{X,\x}(-\log_{0} f)$ and $\delta_{n}$ a Euler-homogeneity for $f_{\red}$. By the trace-adjoint formula of Proposition \ref{prop: main adjoint right-left formula}:
\[
 \sum_{j} c_{j}^{i,j} = - \sum_{e} \partial_{x_{e}} \bullet h_{e,i}  \text{ for } i \neq n; \;
  \sum_{j} c_{j}^{n,j} = - \sum_{e} \partial_{x_{e}} \bullet h_{e,n} + 1 \text{ for } i = n.
\]
Recall $g = g_{1} \cdots g_{u}$. Since $\delta_{i} \bullet f_{\red} = 0$ for $i \leq n-1$ and since $\delta_{n}$ is Euler-homogeneous on $f_{\red}$, $\eqref{matrix- d n plus u dual with all terms}$ simplifies to 
\[
\begin{bmatrix}
\vdots \\
(-1)^{i}( \psi_{(f^{\prime}gf_{\red})^{-1}F,\x}^{-S})(\delta_{i}) \\
\vdots \\
(-1)^{n}( \psi_{(f^{\prime}gf_{\red})^{-1}F,\x}^{-S}(\delta_{n}) \\
\vdots \\
(-1)^{n+v-1} g_{v} \\
\vdots
\end{bmatrix} .
\]
Thus the image of $\Hom_{\D_{X,\x}[S]}(d^{-n-u}, \D_{X,\x}[S])^{\text{left}}$ is $\D_{X,\x}[S] \cdot \theta_{(f^{\prime}g f_{\red})^{-1}F,\x}^{-S} + \D_{X,\x}[S] \cdot I$, proving the proposition for $n \geq 2.$

As for $n=1$, we can assume $f_{\red} = x$ and $\Der_{X,\x}(-\log f_{\red})$ is freely generated by its Euler-homogeneity. Simplifying \eqref{matrix- d n plus u dual with all terms} is then an easy calculation.
\end{proof}

We endow $\Sp_{f^{\prime}F,\x}^{I}$ with a chain co-complex filtration that is based on a construction of Gros and Narv\'aez-Macarro, cf. page 85 of \cite{GrosMacarro}. 
\begin{prop} \label{prop- the chain filtration}
Let $f = f_{1} \cdots f_{r}$ be free, $F=(f_{1},\dots, f_{r})$, and let $f^{\prime}$ and $I$ be as in Definition \ref{def-extended Spencer}. Abbreviate $\Sp_{\theta_{f^{\prime}F,\x}}^{I}$ to $\Sp^{\bullet}.$ Define a filtration $G^{\bullet}$ on $\Sp^{\bullet}$ by 
\begin{align*}
G^{p} \Sp^{-m} = 
    \bigoplus_{j} \left( F_{(0,1,1)}^{p-m+j} \D_{X,\x}[S] \otimes_{\bigO_{X,\x}} \bigwedge^{m-j} \theta_{f^{\prime}F,\x} \wedge \bigwedge^{j} E \right).
\end{align*}
If $\delta_{1}, \dots, \delta_{n}$ is a basis of $\Der_{X,\x}(-\log f)$, then $\gr_{G}(\Sp^{\bullet})$ is isomorphic to the following Koszul co-complex on $\gr_{(0,1,1)}(\D_{X,\x}[S])$:
\begin{equation} \label{eqn- associated graded complex koszul}
K^{\bullet}(\gr_{(0,1,1)}(\psi_{F,\x}(\delta_{1})), \dots, \gr_{(0,1,1)}(\psi_{F,\x}(\delta_{n})), g_{1}, \dots, g_{u}; \gr_{(0,1,1)}(\D_{X,\x}[S])).
\end{equation}
Moreover, $G^{\bullet}$ naturally gives a filtration on $\text{Hom}_{\D_{X,\x}[S]}(\Sp^{\bullet}, \D_{X,\x}[S])^{\text{left}}$ whose associated graded complex is isomorphic to 
\begin{equation} \label{eqn- associated graded complex koszul dual}
K^{\bullet}(\gr_{(0,1,1)}(-\psi_{F,\x}(\delta_{1})), \dots, \gr_{(0,1,1)}(-\psi_{F,\x}(\delta_{n})), g_{1}, \dots, g_{u}; \gr_{(0,1,1)}(\D_{X,\x}[S])).
\end{equation}
\end{prop}
\begin{proof}
That $G^{\bullet}$ is a chain filtration and that the associated graded co-complex is isomorphic to the Koszul complex $\eqref{eqn- associated graded complex koszul}$ follows from the definitions. As for the dual statement, it is enough to note that $\tau$, the standard right-to-left map (cf. Lemma 4.13 of \cite{me}), preserves weight $0$ entries (under the total order filtration) and sends weight 1 entries $ \delta + p(S)$ to $- \delta + p(S) +$ error terms, where $\delta$ is a derivation and both $p(S)$ and the error terms lie in $\bigO_{X,\x}[S]$.
\end{proof}

We now add hypotheses to the settings of Propositions \ref{prop-dual spencer terminal homology} and \ref{prop- the chain filtration}. First, we assume $I = \mathscr{O}_{X,x} \cdot g$ is principal; second, we assume $f$ is not only free but also strongly Euler-homogeneous and Saito-holonomic. This will let us use results from \cite{me}. The filtration $G^{\bullet}$ will demonstrate that $\Sp_{f^{\prime}F}^{g}$ and its dual are resolutions.

\begin{define} \label{def- duality functors}
For $M$ a left $\D_{X,\x}[S]$-module, denote the $\D_{X,\x}[S]$-dual of $M$ by
\[
\mathbb{D}(M) = \text{RHom}_{\D_{X,x}[S]}(M, \D_{X,\x}[S])^{\text{left}}.
\]
\end{define}

\begin{thm} \label{thm- useful duality formula}
Suppose $f = f_{1} \cdots f_{r} \in \bigO_{X}$ is free, strongly Euler-homogeneous, and Saito-holonomic and $f_{\red} \in \bigO_{X,\x}$ is a Euler-homogeneous reduced defining equation for $f$ at $\x$. Let $F = (f_{1}, \dots, f_{r})$, let $f^{\prime} \in \bigO_{X,\x}$ be compatible with $f$, and let $g \in \bigO_{X,\x}$ such that $f \in \bigO_{X,\x} \cdot g$. Then 
\[
\mathbb{D} \left( \frac{\D_{X,\x}[S]f^{\prime}F^{S}}{\D_{X,\x}[S] \cdot g f^{\prime}F^{S}} \right) \simeq \frac{\D_{X,\x}[S] (g f^{\prime} f_{\red})^{-1} F^{-S}}{\D_{X,\x}[S] (f^{\prime} f_{\red})^{-1} F^{-S}}[n+1].
\]

\end{thm}

\begin{proof}

We first show that \eqref{eqn- associated graded complex koszul} and \eqref{eqn- associated graded complex koszul dual} are both resolutions; in fact, showing \eqref{eqn- associated graded complex koszul} is a resolution proves \eqref{eqn- associated graded complex koszul dual} is as well. Let $\delta_{1}, \dots, \delta_{n}$ be a basis of $\Der_{X,\x}(-\log f).$ Since $\gr_{(0,1,1)}(\D_{X,\x}[S])$ is graded local and $\gr_{(0,1,1)}(\psi_{F,\x}(\delta_{i}))$ and $f$ all live in the graded maximal ideal, it is sufficient to prove that the Koszul co-complex \eqref{eqn- associated graded complex koszul} is a resolution after localization at the graded maximal ideal.
By Theorem 2.23 of \cite{me}, $\widetilde{L_{F,\x}}$ is Cohen--Macaulay and prime of dimension $n+r.$ Therefore $\widetilde{L_{F,\x}} + \gr_{(0,1,1)}(\D_{X,\x}[S]) \cdot f$ has dimension $n+r-1$. Moreover, this ideal's dimension does not change after localization at the graded maximal ideal. Theorem 2.1.2 of \cite{BrunsHerzog} then implies \eqref{eqn- associated graded complex koszul} is a resolution after said localization, finishing this part of the proof. 

Since \eqref{eqn- associated graded complex koszul} is a resolution, a standard spectral sequence argument associated to the filtered co-complex of $\Sp_{f^{\prime}F,\x}^{g}$ implies $\Sp_{f^{\prime}F,\x}^{g}$ is a resolution. By Theorem \ref{thm-gen by derivations} and the definition of the augmentation map it resolves $\frac{\D_{X,\x}[S]f^{\prime}F^{S}}{\D_{X,\x}[S]g f^{\prime} F^{S}}.$ Similar reasoning verifies that $\Hom_{\D_{X,\x}[S]}(\Sp_{f^{\prime}F,\x}^{g}, \D_{X,\x}[S])^{\text{left}}$ is a resolution. Because $f_{\red}$ is Euler homogeneous, the claim follows by Proposition \ref{prop-dual spencer terminal homology} and Theorem \ref{thm-gen by derivations -s}.
\end{proof}

\begin{remark}
We are skeptical that \eqref{eqn- associated graded complex koszul} is a resolution for any non-principal, non-pathological $I$. Possible candidates are linear free divisors $f$ with many factors, even though the non-pathological examples in $n \leq 4$ fail, cf. \cite{GrangerMondLinearFreeDivisors}.
\end{remark}

\subsection{Principality of $\sqrt{B_{f^{\prime}F,\x}^{g}}$} \text{ }

Here we discuss the principality of the radical of $B_{f^{\prime}F,\x}^{g}.$ The argument is essentially the same as Proposition 20 of \cite{Maisonobetheory}, but we do not have to appeal to tame pure extensions because of our hypotheses on $f$. 

We will need some homological definitions for modules over non-commutative rings, cf. Appendix IV of \cite{Bjork} for a detailed treatment. We say a $\mathscr{D}_{X,\x}[S]$-module $M$ has \emph{grade} $j$ if $\text{Ext}_{\mathscr{D}_{X,\x}[S]}^{k}(M, \mathscr{D}_{X,\x}[S])$ vanishes for all $k < j$ and is nonzero for $k = j$. We say $M$ is \emph{pure} of grade $j$ if every nonzero submodule of $M$ has grade $j$. We also need the following filtration on $\mathscr{D}_{X,\x}[S]$: 

\begin{define} \label{def- order filtration, characteristic}

Define the \emph{order filtration} $F_{(0,1,0)}$ on $\D_{X}[S]$ by designating, in local coordinates, every $\partial_{x_{k}}$ weight one and every element of $\bigO_{X}[S]$ weight zero. Let $\gr_{(0,1,0)}(\D_{X}[S])$ denote the associated graded object and note that locally $\gr_{(0,1,0)}(\D_{X}[S]) \simeq \bigO_{X}[Y][S]$, with $\gr_{(0,1,0)}(\partial_{x_{k}}) = y_{k}.$ For a coherent $\D_{X}[S]$-module $M$ and any good filtration $\Gamma$ on $M$ relative to $F_{(0,1,0)}$, the \emph{characteristic ideal} $J^{\text{rel}}(M) \subseteq \gr_{(0,1,1)}(\D_{X}[S])$ is defined as 
\[
J^{\text{rel}}(M) = \sqrt{\ann_{\gr_{(0,1,0)}(\D_{X}[S])} \gr_{\Gamma}(M) }
\]
and is independent of the choice of good filtration. 
\end{define}

\begin{prop} \label{prop- principal radical} \text{\normalfont (Compare to Proposition 20 of \cite{Maisonobetheory})}
Suppose $f = f_{1} \dots f_{r} \in \bigO_{X}$ is free, strongly Euler-homogeneous, and Saito-holonomic such that the reduced divisor of $f$ is Euler-homogeneous. Let $F = (f_{1}, \dots, f_{r})$ and select $f^{\prime} \in \mathscr{O}_{X}$ and $g \in \bigO_{X}$ such that $f$ lies in both $\bigO_{X} \cdot f^{\prime}$ and $\bigO_{X} \cdot g$. Then for all $\x$, $\sqrt{B_{f^{\prime}F, \x}^{g}}$ is principal. 
\end{prop}

\begin{proof}
Since $f^{\prime}$ is a section generating a holonomic $\D_{X}$-module, by Proposition 13 of \cite{Maisonobetheory} there is a conical Lagrangian variety $\Lambda \subseteq T^{\star}X$ so that $\V(J^{\text{rel}}( \D_{X}[S]f^{\prime}F^{S})) = \Lambda \times C^{r}.$ So $\V(J^{\text{rel}}(\frac{\D_{X}[S]f^{\prime}F^{S}}{\D_{X}[S] g f^{\prime}F^{S}})) \subseteq \Lambda \times \mathbb{C}^{r}$, that is, in the language of Maisonobe, $\frac{\D_{X}[S]f^{\prime}F^{S}}{\D_{X}[S] g f^{\prime}F^{S}}$ is \emph{major\'e par une Lagrangian}. By Proposition 8 of \cite{Maisonobetheory}, there exist conical Lagrangians $T_{X_{\alpha}}^{\star}X$ and algebraic varieties $S_{\alpha} \subseteq \mathbb{C}^{r}$ such that
\begin{equation} \label{eqn- characteristic variety structure}
\V \left( J^{\text{rel}}(\frac{\D_{X}[S]f^{\prime}F^{S}}{\D_{X}[S] g f^{\prime}F^{S}}) \right) = \cup_{\alpha} T_{X_{\alpha}}^{\star}X \times S_{\alpha}.
\end{equation}
By Proposition 9 of \cite{Maisonobetheory}, $\V(B_{f^{\prime}F,\x}^{g}) = \cup_{\x \in X_{\alpha}} S_{\alpha}.$

Now to show the radical of $B_{f^{\prime}F,\x}^{g}$ is principal, it suffices to show $S_{\alpha}$ is of dimension $r-1$ for each $\alpha$ such that $\x \in X_{\alpha}$; that is, by the description of $T_{X_{\alpha}}^{\star}X$, it suffices to show $J^{\text{rel}}(\frac{\D_{X,\x}[S]f^{\prime}F^{S}}{\D_{X,\x}[S]gf^{\prime}F^{S}})$ is equidimensional of dimension $n+r-1$. By Theorem \ref{thm- useful duality formula}, $\frac{\D_{X,\x}[S]f^{\prime}F^{S}}{\D_{X,\x}[S]gf^{\prime}F^{S}}$ has grade $n+1$. Using Theorem \ref{thm- useful duality formula} again and the characterization of pure modules in terms of double Ext modules, cf. Proposition IV.2.6 of \cite{Bjork}, we deduce $\frac{\D_{X,\x}[S]f^{\prime}F^{S}}{\D_{X,\x}[S]gf^{\prime}F^{S}}$ is a pure $\D_{X,\x}[S]$-module of grade $n+1$. By Theorem IV.5.2 of \cite{Bjork}, $J^{\text{rel}}(\frac{\D_{X,\x}[S]f^{\prime}F^{S}}{\D_{X,\x}[S]gf^{\prime}F^{S}})$ is equidimensional and every minimal prime of the characteristic ideal has codimension $n+1$, completing the proof.
\end{proof}

The next proposition lays out a criterion for $B_{f^{\prime}F,\x}^{g}$ to be principal. The argument is that of the last paragraph of Theorem 2 of \cite{MaisonobeFree}.

\begin{prop} \label{prop- criterion for principal b-ideal} \text{\normalfont (Compare to Theorem 2 of \cite{MaisonobeFree})}
Let $f$, $F$, $f^{\prime}$, and $g$ be as in Proposition \ref{prop- principal radical} and suppose that $\sqrt{B_{f^{\prime}F,\x}^{g}} = \mathbb{C}[S] \cdot b(S)$, i.e. it is principal. Suppose that $(B_{f^{\prime}F,\x} : \sqrt{B_{f^{\prime}F,\x}})$ contains a polynomial $a(S)$ such that $\V(\mathbb{C}[S] \cdot b(S)) \cap \V(\mathbb{C}[S] \cdot a(S))$ has irreducible components of dimension at most $r-2$. Then $B_{f^{\prime}F,\x}^{g}$ equals its radical and is principal.
\end{prop}

\begin{proof}
It suffices to show $b(S)\frac{\D_{X,\x}[S]f^{\prime}F^{S}}{\D_{X,\x}[S]gf^{\prime}F^{S}}$ is zero. If it is nonzero, it is a submodule of the pure module $\frac{\D_{X,\x}[S]f^{\prime}F^{S}}{\D_{X,\x}[S]gf^{\prime}F^{S}}$ of grade $n+1$ and so is itself pure of the same grade. Reasoning as in Proposition \ref{prop- principal radical}, cf. Proposition 9 of \cite{Maisonobetheory} in particular, all the minimal primes of $\mathbb{C}[S]$-annihilator of $b(S) \frac{\D_{X,\x}[S]f^{\prime}F^{S}}{\D_{X,\x}[S]gf^{\prime}F^{S}}$ have dimension $r-1$. But the variety of this annihilator is contained inside $\V(\mathbb{C}[S] \cdot b(S)) \cap \V(\mathbb{C}[S] \cdot a(S))$ which is of dimension $r-2$ by hypothesis. As this is impossible, $b(S) \frac{\D_{X,\x}[S]f^{\prime}F^{S}}{\D_{X,\x}[S]gf^{\prime}F^{S}}$ must be zero.
\end{proof}

\subsection{Symmetry of Some Bernstein--Sato Varieties} \text{ }

As Theorem \ref{thm- useful duality formula} generalizes Corollary 3.6 of \cite{MAcarroDuality} and Proposition 6 of \cite{MaisonobeFree}, one would hope $B_{f^{\prime}F,\x}^{g}$ has a symmetry generalizing Theorem 4.1 of \cite{MAcarroDuality} and Proposition 8 of \cite{MaisonobeFree}. However, without reducedness and with the addition of $f^{\prime}$, symmetry seems to depend on the factorization of $f$.
\begin{define} \label{def-unmixed}
Suppose $f$ has a factorization into irreducibles $l_{1}^{v_{1}} \cdots l_{q}^{v_{q}}$ at $\x$ where the $l_{t}$ are distinct and $v_{t} \in \mathbb{Z}_{+}$. Let $f = f_{1} \cdots f_{r}$ be some other factorization of $f$ and let $F = (f_{1}, \dots, f_{r})$. We say the factorization $f = f_{1} \cdots f_{r}$ is \emph{unmixed} if the following hold:
\begin{enumerate}[(i)]
\item for each $k$, there exists $d_{k} \in \mathbb{Z}_{+}$ and $J_{k} \subseteq [q]$ such that $f_{k} = \prod_{j \in J_{k}} l_{j}^{d_{k}};$ 
\item if $i, j \in J_{k}$, then $v_{i} = v_{j}$.
\end{enumerate}
$F$ is \emph{unmixed} when it corresponds to an unmixed factorization; $F$ is \emph{unmixed up to units} if there exists units $u_{1}, \dots, u_{r}$ such that $uF = (u_{1}f_{1}, \dots, u_{r}f_{r})$ is unmixed. Given an unmixed factorization, let the $\emph{repeated multiplicity}$ of $F$ be $\{m_{k}\}_{k}$ where, for any $j \in J_{k}$ (and thus all), $m_{k}$ is the multiplicity of $l_{j}$ with respect to $f$. 

For $f^{\prime} \in \bigO_{X,\x}$ compatible with $f$, we say $(f^{\prime}, F)$ is an \emph{unmixed pair} if:
\begin{enumerate}[(i)$'$]
    \item $F$ is unmixed;
    \item $f^{\prime} = \prod_{k} \prod_{j \in J_{k}} l_{j}^{d_{k}^{\prime}}$ for $ d_{k}^{\prime} \in \mathbb{Z}$.

\end{enumerate}
The pair $(f^{\prime},F)$ is an \emph{unmixed pair up to units} if $F$ is unmixed up to units and $f^{\prime}$ satisfies (ii$'$) after possibly multiplying by a unit. For $(f^{\prime}, F)$ an unmixed pair up to units, the \emph{pairs of repeated powers} of $(f^{\prime},F)$ are $\{(d_{k}^{\prime}, d_{k})\}_{k}$.
\end{define}

\begin{lemma} \label{lemma- unmixed}
Write $f = l_{1}^{v_{1}} \cdots l_{q}^{v_{q}}$ where the $l_{i}$ are distinct and irreducible; $f_{k} = \prod_{j \in J_{k}} l_{j}^{d_{k}}$; $f_{\red} = l_{1} \cdots l_{q}$. Assume that $f^{\prime}$ and $g$ are compatible with $f$, $F = (f_{1}, \dots, f_{r})$ a factorization of $f$, $(f^{\prime}, F)$  and $(g, F)$ are unmixed pairs with pairs of repeated powers $\{(d_{k}^{\prime}, d_{k})\}_{k}$ and $\{(d_{k}^{\prime \prime}, d_{k})\}_{k}$, and $\{m_{k}\}_{k}$ the repeated multiplicities of $F$. If $\varphi: \mathbb{C}[S] \to \mathbb{C}[S]$ is the automorphism of $\mathbb{C}$-algebras induced by
\[
\varphi(s_{k}) = - s_{k} - \frac{1}{m_{k}} - \frac{2d_{k}^{\prime}}{d_{k}} - \frac{d_{k}^{\prime \prime}}{d_{k}},
\]
then for $\delta \in \Der_{X,\x}(-\log f)$, and after extending $\varphi$ to $\D_{X,\x}[S]$,
\[
\varphi(\psi_{(f^{\prime}gf_{\red})^{-1}F,\x}^{-S}(\delta)) = \psi_{f^{\prime}F,\x}^{S}(\delta).
\]

\end{lemma}

\begin{proof}
This is a straightforward computation once we observe that $v_{j}$ is the sum of all the $d_{k}$ such that $l_{j}$ divides $f_{k}$.
\end{proof}

\begin{thm} \label{thm-unmixed symmetry}
Suppose $f = f_{1} \cdots f_{r} \in \bigO_{X}$ is free, strongly Euler-homogeneous, and Saito-holonomic, and while $f$ is not necessarily reduced, suppose that it admits a strongly Euler-homogeneous reduced defining equation at $\x$. Let $F = (f_{1}, \cdots, f_{r})$ and select $g \in \bigO_{X,\x}$ such that $f \in \bigO_{X,\x} \cdot g$. Assume that $f^{\prime}$ and $g$ are compatible with $f$, $(f^{\prime}, F)$  and $(g, F)$ are unmixed pairs up to units with pairs of repeated powers $\{(d_{k}^{\prime}, d_{k})\}_{k}$ and $\{(d_{k}^{\prime \prime}, d_{k})\}_{k}$, and $\{m_{k}\}_{k}$ are the repeated multiplicities of $F$. If $\varphi: \mathbb{C}[S] \to \mathbb{C}[S]$ is the automorphism of $\mathbb{C}$-algebras induced by
\[
\varphi(s_{k}) = - s_{k} - \frac{1}{m_{k}} - \frac{2d_{k}^{\prime}}{d_{k}} - \frac{d_{k}^{\prime \prime}}{d_{k}},
\]
then
\[
B(S) \in B_{f^{\prime}F,\x}^{g} \iff \varphi(B(S)) \in B_{f^{\prime}F,\x}^{g}.
\]
\end{thm}

\begin{proof}
We first reduce to the case that $(f^{\prime}, F)$ and $(g, F)$ are unmixed pairs. It follows from the functional equation that if $u$ is a unit in $\mathscr{O}_{X,\x}$, then $B_{f^{\prime}F,\x}^{g} = B_{uf^{\prime}F,\x}^{g}$ and $B_{f^{\prime}F,\x}^{g} = B_{f^{\prime}F,\x}^{ug}$. To finish the reduction, we must also verify that if $F^{\prime} = (u_{1}f_{1}, \dots, u_{r}f_{r})$ for units $u_{1}, \dots, u_{r}$ in $\mathscr{O}_{X,\x}$, then $B_{f^{\prime}F,\x}^{g} = B_{f^{\prime}F^{\prime},\x}^{g}$. This follows by arguing as in Lemma 10 (i) of \cite{BahloulOaku} wherein the claim is proved for $f^{\prime} = 1$ and $g = f$. 

By the $\mathbb{C}[S]$-linearity of $\mathbb{D}$, cf. Remark 3.2 of \cite{MAcarroDuality}, and by Theorem \ref{thm- useful duality formula}, 
\[
B(S) \in \ann_{\mathbb{C}[S]} \frac{\D_{X,\x}[S]f^{\prime}F^{S}}{\D_{X,\x}[S] \cdot g f^{\prime}F^{S}}
\implies B(S) \in \ann_{\mathbb{C}[S]} \frac{\D_{X,\x}[S](g f^{\prime} f_{\red})^{-1}F^{-S}}{\D_{X,\x}[S] \cdot (f^{\prime} f_{\red})F^{-S}}
\]
where we may assume $f_{\red}$ is as in Lemma \ref{lemma- unmixed}, cf. Remark \ref{rmk-strongly Euler}. In other words, 
\begin{align*}
B(S) \in \mathbb{C}[S]& \cap ( \D_{X,\x}[S] \cdot \theta_{f^{\prime} F,\x} + \D_{X,\x}[S] \cdot g) \\
    & \implies B(S) \in \mathbb{C}[S] \cap (\D_{X,\x}[S] \cdot \theta_{(f^{\prime}gf_{\red})^{-1}F,\x}^{-S} + \D_{X,\x}[S] \cdot g).
\end{align*}
By Lemma \ref{lemma- unmixed}, $\varphi$ induces a $\D_{X,\x}$-automorphism that sends $\D_{X,\x}[S] \cdot \theta_{(f^{\prime}gf_{red})^{-1}F,\x}^{-S} + \D_{X,\x}[S] \cdot g$ to $\D_{X,\x}[S] \cdot \theta_{f^{\prime}F,\x} + \D_{X,\x}[S] \cdot g.$ Therefore $\varphi(B_{f^{\prime}F,\x}^{I}) \subseteq B_{f^{\prime}F,\x}^{I}.$ The reverse containment follows from the fact $\varphi$ is an involution.
\end{proof}

\begin{remark}
Suppose $f$, $f^{\prime}$, and $F$ are as in Theorem \ref{thm-unmixed symmetry}, and $I$ is the ideal generated by $g_{1}, \dots, g_{u}$ such that $f \in \bigO_{X,\x} \cdot g_{j}$. If $\Sp_{f^{\prime}F,\x}^{g}$ and its $\D_{X,\x}[S]$-dual are both resolutions, then $\varphi$ fixes $B_{f^{\prime}F,\x}^{I}$. Note that $\varphi$ depends only on the product of the $g_{j}$.
\end{remark}

Let us catalogue some of the most useful versions of the theorem:

\begin{cor} \label{cor- symmetry formula list}
Suppose $f = f_{1} \cdots f_{r} \in \bigO_{X}$ is free, strongly Euler-homogeneous, and Saito-holonomic, and while $f$ is not necessarily reduced, suppose that it admits a strongly Euler-homogeneous reduced defining equation at $\x$.  Let $F= (f_{1}, \dots, f_{r})$ and $\varphi$ be as in Theorem \ref{thm-unmixed symmetry}.
\begin{enumerate}[(a)]
    \item Suppose that $F = (l_{1}, \dots, l_{1}, \dots, l_{q})$ with each $l_{t}$ appearing $v_{t}$ times, and $f^{\prime}$ and $g$ any elements of $\bigO_{X,\x}$ dividing $f$. Then $\varphi(B_{f^{\prime}F,\x}^{g}) = B_{f^{\prime}F,\x}^{g}$.
    \item Suppose $f$ is reduced, $F$ corresponds to any factorization, $f^{\prime} = \prod_{k^{\prime} \in K^{\prime}} f_{k}^{\prime}$, $g = \prod_{k \in K} f_{k}$, for $K^{\prime}, K \subseteq [r]$. Then $\varphi(B_{f^{\prime}F,\x}^{g}) = B_{f^{\prime}F,\x}^{g}.$
    \item Suppose $f^{\prime}$ divides $f = f_{1} \cdots f_{r}$, $F = (f_{1}, \dots, f_{r})$ and $g= \frac{f}{f^{\prime}}$. If $(f^{\prime}, F)$ is an unmixed pair up to units, then $\varphi(B_{f^{\prime}F,\x}^{g}) = B_{f^{\prime}F,\x}^{g}$ 
    \item Suppose $f = f_{\red}^{k}$ and $F = (f_{\red}^{k})$. Then $\varphi(s) = -s - 1 - \frac{1}{k}$ and $\varphi(B_{f^{k},\x}) = B_{f^{k},\x}.$
\end{enumerate}
\end{cor}

\begin{proof}
All that must be checked is that the appropriate things are unmixed pairs up to units. For example, in (a) and (b), $F$ is unmixed up to units because it is a factorization into irreducibles, possibly with repetition, and because $f$ is reduced, respectively. In both cases, $d_{k}$, $d_{k}^{\prime}$, and $d_{k}^{\prime \prime}$ are all $1$.
\end{proof}

The symmetry property for the Bernstein--Sato polynomial of a reduced divisor forces all its roots to lie inside $(-2,0)$, cf. \cite{MAcarroDuality}. We have the following generalization for powers of reduced divisor:
\begin{cor}
Suppose $f$ is reduced, free, strongly Euler-homogeneous, and Saito-holonomic. Then $\V(B_{f^{k}}) \subseteq (-1-\frac{1}{k}, 0)$. If $b_{f^{k}, \min}$ is the smallest root of the Bernstein--Sato polynomial of $f^{k}$, then $b_{f^{k}, \min} \to -1$ as $k \to \infty.$
\end{cor}

\begin{proof}
Since freeness, strongly Euler-homogeneous, and Saito-holonomicity pass from $f_{\red}$ to $f^{k}$ we may use Corollary \ref{cor- symmetry formula list} to improve the well known containment $\V(B_{f^{k}, \x}) \subseteq (-\infty, 0)$ to $\V(B_{f^{k}, \x}) \subseteq (-1 - \frac{1}{k}, 0)$. The rest follows since $-1 \in \V(B_{f^{k},\x}).$
\end{proof}

\section{Bernstein--Sato Varieties for Tame and Free Arrangements}

In this section we study the global Bernstein--Sato ideals $B_{f^{\prime}F}^{g}$ where $f$ is a central, not necessarily reduced, tame hyperplane arrangement, $f^{\prime}$ divides $f$, $g = \frac{f}{f^{\prime}}$, and $F$ corresponds to the factorization $f=f_{1} \cdots f_{r}$, which need not be into linear forms. We always assume $\bigO_{X,\x} \cdot f^{\prime} \neq \bigO_{X,\x} \cdot f$. We revisit the arguments of Maisonobe in \cite{MaisonobeFree} giving full details for our versions of Lemma 2 and Proposition 9 in the first subsection and Proposition 10 in the second. We generalize the strategy of Lemma 2 and Proposition 9 to compute a principal ideal containing $B_{f^{\prime}F}^{g}$ for tame hyperplane arrangements and any $F$; we generalize Proposition 10 to find an element of $B_{f^{\prime}F}^{g}$ when $f$ is not necessarily reduced, not necessarily tame, and $F$ is the total factorization of $f$ into linear forms. As Maisonobe does in Theorem 2 of loc. cit., in the third subsection we use the symmetry of $B_{f^{\prime}F}^{g}$ when $f$ is free and $(f^{\prime},F)$ is an unmixed pair up to units to provide rather precise estimates of $\V(B_{f^{\prime}F}^{g})$. In certain situations, these estimates compute $\V(B_{f^{\prime}F}^{g})$.

\begin{define} \label{def-hyperplane arrangement factorizations}

Let $f \in \mathbb{C}[x_{1}, \dots, x_{n}]$ be a central, not necessarily reduced, hyperplane arrangement of degree $d$ whose factorization into homogeneous linear forms is $f = l_{1} \cdots l_{d}$. Associated to $f$ is the intersection lattice $L(A)$, partially ordered by reverse inclusion and with smallest element $\mathbb{C}^{n}$. We call any $X \in L(A)$ an \emph{edge} of $L(A)$. The \emph{rank} of $X$ is the length of a maximal chain in $L(A)$ with smallest element $\mathbb{C}^{n}$ and largest element $X$. We denote the rank of $X$ by $r(X)$; for example, $r(\V(l_{i})) = 1$. Given an edge $X \in L(A)$ we define $J(X)$ to be the subset of $[d]$ identifying the hyperplanes that contain $X$, that is:
\[
X = \bigcap\limits_{j \in J(X)} \V(l_{j}).
\]
Note that because $f$ is not necessarily reduced $J(X)$ may contain indices $i$ and $j$ such that $\V(l_{i}) = \V(l_{j}).$ Given an edge $X$, there is the subarrangement $A_{X}$ which has the defining equation
\[
f_{X} = \prod\limits_{j \in J(X)} l_{j}.
\]
The degree of $f_{X}$ is denoted $d_{X}$. So $d_{X} = \abs{J(X)}.$ The edge $X$ is \emph{decomposable} if there is a change of coordinates $y_{1} \sqcup y_{2}$, $y_{1}$ and $y_{2}$ disjoint, such that $f_{X} = p q$ where $p$ and $q$ are hyperplane arrangements using variables only from $y_{1}$ and $y_{2}$ respectively. Otherwise $X$ is \emph{indecomposable}.

Consider a potentially different factorization $f = f_{1} \cdots f_{r}$ where each $f_{k}$ is of degree $d_{k}$. Since each $f_{k}$ is a product of some of the $l_{m}$, let $S_{k} \subseteq [d]$ identify the linear forms comprising $f_{k}$, that is, 
\[
f_{k} = \prod_{m \in S_{k}} l_{m}.
\]
The factorization $f = f_{1} \cdots f_{r}$ induces a factorization of $f_{X}$. Define $S_{X,k} \subseteq [d]$ by
\[
S_{X,k} = J_{X} \cap S_{k}.
\]
Then $f_{X}$ inherits the factorization $f_{X} = f_{X,1} \cdots f_{X,r}$ where 
\[
f_{X,k} = \prod\limits_{j \in S_{X,k}} l_{j}.
\]
We say $f_{X,k}$ has degree $d_{X,k}$. We also write $F_{X} = (f_{X,1}, \dots, f_{X,r})$.

Any hyperplane arrangement has a reduced equation $f_{\red}$ of degree $d_{\red}$. We define $f_{X,\red}$, $d_{X,\red}$, $f_{X,k,\red}$, and $d_{X,k,\red}$ similarly.

If $f^{\prime}$ of degree $d^{\prime}$ divides $f$, then all the previous constructions apply to $f^{\prime}$. Define $f_{\red}^{\prime}, d_{\red}^{\prime}, f_{X}^{\prime}, d_{X}^{\prime}, f_{X,\red}^{\prime}, d_{X,\red}^{\prime}, f_{X,k}^{\prime}, d_{X,k}^{\prime}, f_{X,k,\red}^{\prime}, d_{X,k,\red}^{\prime}$ in the natural ways. 

\end{define}

We will be working with the Weyl algebra $\A = \mathbb{C}[x_{1}, \dots, x_{n}, \partial_{1}, \dots, \partial_{n}]$ where the \emph{global Bernstein--Sato ideal} $B_{f^{\prime}F}^{g}$ is defined similarly to $B_{f^{\prime}F,\x}^{g}$ except using $\A[S]$ operators. Write $B_{f^{\prime}f}^{g}$ when $F=(f)$ corresponds to the trivial factorization $f = f.$  We use the notation $\theta_{f^{\prime}F}$ and $\psi_{f^{\prime}F}$ for the algebraic, global versions of $\theta_{f^{\prime}F,\x}$ and $\psi_{f^{\prime}F,\x}.$ 

By Corollary \ref{cor-algebraic category} and Examples \ref{ex- hyperplane saito} and \ref{ex-hyperplane strongly Euler}, if $f$ is tame and $f^{\prime}$ divides $f$, then $\ann_{\A[S]} f^{\prime}F^{S}$ is generated by derivations. Moreover, $f_{\red}$ is strongly Euler-homogeneous itself. Finally, since $f$ is central, the $\mathbb{C}^{\star}$-action on $\V(f)$ can be used to show $B_{f^{\prime}F}^{g} = B_{f^{\prime}F,0}^{g}.$ Therefore we can apply the results of the previous sections. 

Finally, recall that for any central hyperplane arrangement $f \in \mathbb{C}[x_{1}, \dots, x_{n}]$ of degree $d$, the \emph{Euler derivation} $E = x_{1}\partial_{1} + \cdots x_{n} \partial_{n}$ satisfies $E \bullet f = d f$. Thus $\frac{1}{d} E$ is a strong Euler-homogeneity for $f$ at the origin. 

\subsection{An Ideal Containing $B_{f^{\prime}F}^{g}$} \text{ }

We compute a principal ideal containing $B_{f^{\prime}F}^{g}$ where $f$ is a central, indecomposable, and tame hyperplane arrangement, $f^{\prime}$ divides $f$, $g = \frac{f}{f^{\prime}}$, and $F$ corresponds to any factorization. The argument tracks Lemma 2 and Proposition 9 of \cite{MaisonobeFree} but we have replaced freeness with tameness, reduced with non-reduced, added $f^{\prime}$, and we will use any factorization $F$ instead of the factorization into linear forms. Though the approach is similar to Maisonobe's, we provide detail for the sake of the reader.

\begin{define} \label{def-right normal form}
The \emph{right normal form} of $P \in \A[S]$ is the unique expression
\[
P = \sum_{\textbf{u}} \partial^{\textbf{u}} P_{\textbf{u}}
\]
where $P_{\textbf{u}} \in \mathbb{C}[X][S].$ The \emph{right constant term} of $P$ is $P_{\textbf{0}}$. Note that for $P, Q \in \A[S]$, the right constant term of $P + Q$ is $P_{\textbf{0}} + Q_{\textbf{0}}.$
\end{define}

\begin{convention}
Let $\mathbb{C}[X]_{t}$ be the subspace of homogeneous polynomials in $\mathbb{C}[X]$ of degree $t$ and let $\mathbb{C}[X]_{\geq t}$ be the ideal of $\mathbb{C}[X]$ generated by the homogeneous polynomials of degree at least $t$. Denote by $\mathbb{C}[X]_{t}[S]$ and $\mathbb{C}[X]_{\geq t}[S]$ the $\mathbb{C}[S]$-modules generated by $\mathbb{C}[X]_{t}$ and $\mathbb{C}[X]_{\geq t}$ respectively.
\end{convention}

\begin{lemma} \label{lemma constant terms calculation}

Consider a derivation $\delta = \sum_{i} a_{i} \partial_{x_{i}}$ and a polynomial $c \in \mathbb{C}[X][S].$ If $P \in A_{n}(\mathbb{C})[S]$ has right constant term $P_{\textbf{0}}$, then $P \cdot (\delta - c)$ has right constant term 
\[
- (\sum_{i} \partial_{x_{i}} \bullet a_{i}) P_{\textbf{0}} - \delta \bullet (P_{\textbf{0}}) -c P_{\textbf{0}}.
\]

\end{lemma}

\begin{proof}

Consider the right normal form $\sum \partial^{\textbf{u}} P_{\textbf{u}}$ of $P$. Then 
\begin{align*}
    P \cdot (\delta -c) 
    &= \sum\limits_{\textbf{u}} \partial^{\textbf{u}} (\delta P_{\textbf{u}} - \delta \bullet P_{\textbf{u}} - P_{\textbf{u}} c) \\
    &= \sum\limits_{\textbf{u}} \partial^{\textbf{u}} ((\sum_{i} \partial_{i} a_{i} - \sum_{i} \partial_{i} \bullet a_{i}) P_{\textbf{u}} - \delta \bullet P_{\textbf{u}} - P_{\textbf{u}} c) \\
    &= \sum\limits_{\textbf{u}} \partial^{\textbf{u}} \sum_{i} \partial_{i} a_{i} P_{\textbf{u}} + \sum\limits_{\textbf{u}} \partial^{\textbf{u}}(( - \sum_{i} \partial_{i} \bullet a_{i}) P_{\textbf{u}} - \delta \bullet (P_{\textbf{u}}) - c P_{\textbf{u}}).
\end{align*}
Because $\sum_{\textbf{u}} \partial^{\textbf{u}} \sum_{i} \partial_{i} a_{i} P_{\textbf{u}}$ has constant term $0$, the lemma follows. 
\end{proof}

\begin{lemma} \label{lemma-constant term syzygy}

Suppose $\delta \in \Der_{X}(-\log f)$ can be written as $\sum_{i=1}^{n} a_{i} \partial_{i}$ where each $a_{i}$ is a homogeneous polynomial of degree $t$ in $\mathbb{C}[X].$ Let $f = f_{1} \cdots f_{r}$ where each $f_{k}$ is homogeneous, $F = (f_{1}, \dots, f_{r})$, and $f^{\prime}$ is a homogeneous polynomial dividing $f$. If $P \in A_{n}(\mathbb{C})[S]$, then the right constant term of $P \cdot \psi_{f^{\prime}F}(\delta)$ lies in $\mathbb{C}[X]_{\geq t-1}[S]$.
\end{lemma}

\begin{proof}

Recall  $\psi_{f^{\prime}F}(\delta) = \delta - \sum \frac{\delta \bullet f_{k}}{f_{k}} s_{k} - \frac{\delta \bullet f^{\prime}}{f^{\prime}}$. By the choice of $\delta$, $-\sum_{k=1}^{r} \frac{\delta \bullet f_{k}}{f_{k}}s_{k} - \frac{\delta \bullet f^{\prime}}{f^{\prime}} \in \mathbb{C}[X]_{t-1}[S].$ By Lemma \ref{lemma constant terms calculation}, the right constant term of $P \cdot \psi_{F}(\delta)$ is 
\[
( - \sum_{i} \partial_{i} \bullet a_{i}) P_{0} - \delta \bullet P_{0} - (\sum_{k} \frac{\delta \bullet f_{k}}{f_{k}}s_{k}) P_{0} - \frac{\delta \bullet f^{\prime}}{f^{\prime}} P_{\textbf{0}}.
\]
Let $m$ be the smallest nonnegative integer such that $P_{0} \in \mathbb{C}[X]_{\geq m}[S].$ Because $\partial_{i} \bullet a_{i} \in \mathbb{C}[X]_{t-1}$ and $\delta \bullet P_{0} \in \mathbb{C}[X]_{\geq t + m - 1}[S]$ the claim follows.
\end{proof}

There is a natural $\mathbb{C}[X]$-isomorphism between $\Der_{X}(-\log_{0} f)$ and the first syzygies of the Jacobian ideal $J(f)$, i.e. the ideal of $\mathbb{C}[X]$ generated by the partials of $f$. If $f$ is homogeneous, so is $J(f)$ and so is its first syzygy module. 

\begin{define} For $f$ homogeneous, define $\mdr(f)$ to be
\[
\mdr(f) = \text{min}  \{ t \mid  \text{ there exists a homogeneous syzygy of } J(f) \text{ of degree } t \}.
\]

\end{define}

\begin{remark} \label{rmk-mdr, indecomposable}
\begin{enumerate}[(a)]
\item It known that a central hyperplane arrangement of $f$ of rank $\geq 2$ is indecomposable if and only if $\mdr(f) \geq 2.$ For one direction use the first part of Theorem 5.13 of \cite{uli}; for the other, use the two disjoint Euler derivations induced by the coordinate change.
\item Identify $\Der_{X}(-\log_{0} f)$ and first syzygies of $J(f)$ to conclude that we may pick a generating set $\delta_{1}, \dots, \delta_{m}$ of $\Der_{X}(-\log_{0} f)$ such that $\delta_{j} = \sum_{i=1}^{r} a_{j,i} \partial_{i}$ and each $a_{j,i} \in \mathbb{C}[X]$ is homogeneous of degree at least $\mdr(f).$
\end{enumerate}
\end{remark}

We can now prove our version of Lemma 2 from \cite{MaisonobeFree}. The argument is similar but we defer applying any symmetry of $B_{f^{\prime}F}^{g}$ until later.  

\begin{thm} \label{thm-first subset thm} \text{\normalfont (Compare to Lemma 2 in \cite{MaisonobeFree})}
Let $f$ be a central, not necessarily reduced, indecomposable and tame hyperplane arrangement of rank $n \geq 2$ and let $F = (f_{1}, \dots, f_{r})$ correspond to any factorization $f = f_{1} \cdots f_{r}$. If $f^{\prime}$ divides $f$ and $g = \frac{f}{f^{\prime}}$, then 
\[
B_{f^{\prime}F}^{g} \subseteq \mathbb{C}[S] \cdot \prod_{j=0}^{\mdr(f) + d-d^{\prime} - 3} \left( \sum_{k} d_{k} s_{k} + n + d^{\prime} + j \right).
\]

\end{thm}

\begin{proof}

To begin, we choose two polynomials. First fix $0 \neq B(S) \in B_{f^{\prime}F}^{g}$. By definition of $B_{f^{\prime}F,\x}^{g}$, the polynomial $B(S)$ lies in $\ann_{\A[S]} f^{\prime}F + \A[S] \cdot g$. Second, pick a nonzero homogeneous polynomial $v \in \mathbb{C}[X]$ such that (i) $\degree(v) \leq \mdr(f) - 2$ and (ii) there exists a point $\alpha \in \V( g) \setminus \V(v)$. By Remark \ref{rmk-mdr, indecomposable} such a choice of $v$ is possible. Note that $v B(S) \in \ann_{\A[S]} f^{\prime}F + \A[S] \cdot g$.

Let $\delta_{1}, \dots, \delta_{m}$ generate $\Der_{X,\x}(-\log_{0} f)$ where $\delta_{j} = \sum_{j} a_{j,i} \partial_{i}$; let $E$ by the Euler derivation. By Remark \ref{rmk-mdr, indecomposable}, we may assume $\{a_{j,i}\}_{i}$ are all homogeneous polynomials of the same degree where that degree is at least $\mdr(f)$. Corollary \ref{cor-algebraic category} implies there exist $L, P, Q_{2}, \dots, Q_{m} \in \A[S]$ such that 
\begin{equation} \label{eqn-third eqn main thm star}
    v B(S) =  L g + P \psi_{f^{\prime}F}(E) + \sum_{j = 2}^{m} Q_{j} \psi_{f^{\prime}F}(\delta_{j}).
\end{equation}

Express both sides of (\ref{eqn-third eqn main thm star}) in their right normal form. First consider the right hand side of (\ref{eqn-third eqn main thm star}). By Lemma \ref{lemma-constant term syzygy}, the right constant term of $Q_{j}\psi_{f^{\prime} F}(\delta_{j})$ is in $\mathbb{C}[X]_{\geq \mdr(f) -1}[S]$. Write the right constant term $L_{\textbf{0}}$ of $L$ as $L_{\textbf{0}} = \sum_{t} L_{\textbf{0}}^{t}$ where $L_{\textbf{0}}^{t} \in \mathbb{C}[X]_{t}[S]$; similarly, write the right constant term $P_{\textbf{0}}$ of $P$ as $P_{\textbf{0}} = \sum_{t} P_{\textbf{0}}^{t}$ where $P_{\textbf{0}}^{t} \in \mathbb{C}[X]_{t}[S].$ The right constant term of $L g$ is $L_{\textbf{0}} g$. By Lemma \ref{lemma constant terms calculation}, the right constant term of $P \psi_{f^{\prime} F}(E)$ is 
\begin{align*}
\sum_{t} -n P_{\textbf{0}}^{t} & - E \bullet P_{\textbf{0}}^{t}  - \left( \sum_{k} \frac{E \bullet f_{k}}{f_{k}} s_{k} \right) P_{\textbf{0}}^{t} - \frac{E \bullet f^{\prime}}{f^{\prime}} P_{\textbf{0}}^{t} \\
    & = \sum_{t} \left( -n - t -  \sum_{k} d_{k}s_{k} - d^{\prime} \right) P_{\textbf{0}}^{t}.
\end{align*}

On the other hand, the right constant term of $v B(S)$ is $v B(S)$ itself. Note that $v B(S) \in \mathbb{C}[X]_{\degree(v)}[S]$ and, by the choice of $v$, $\degree(v) < \mdr(f) -1.$ So when we write the right constant term of both sides of (\ref{eqn-third eqn main thm star}), the left hand side is $ v B(S)$ and the right hand side can be written using only terms in $\mathbb{C}[X]_{\degree(v)}[S].$ We deduce
\begin{equation} \label{eqn-constant term formula}
    v B(S) = L_{\textbf{0}}^{\degree(v)} g + 
    \left( -n - \degree(v) - d^{\prime} - \sum_{k} d_{k} s_{k} \right) P_{\textbf{0}}^{\degree(v)}.
\end{equation}

The equation (\ref{eqn-constant term formula}) occurs in $\mathbb{C}[X]_{\text{deg}(v)}[S]$ and so the equality is still true when regarding all the elements as belonging to $\mathbb{C}[X][S]$. By the choice of $v$, there exists $\alpha \in \V ( g) \setminus  \V (v)$. The polynomial $P_{\textbf{0}}^{\degree(v)}$ cannot vanish at $\alpha$, lest $B(S) = 0.$ By evaluating (\ref{eqn-constant term formula}) at $\alpha$ we see 
\begin{equation} \label{eqn-divisibilty statement}
B(S) \in \mathbb{C}[S] \cdot \left( -n - \degree(v) - d^{\prime} - \sum_{k} d_{k} s_{k} \right) .
\end{equation}
As $\degree(v)$ is flexible, 
\begin{equation} \label{eqn - first approximate BS ideal containment}
B_{f^{\prime}F, \x}^{g} \subseteq \mathbb{C}[S] \cdot \prod_{j=0}^{\mdr(f)-2} \left( \sum_{k} d_{k} s_{k} + n + d^{\prime} + j \right).
\end{equation}

Now suppose $(f) \subseteq (f^{\prime \prime}) \subseteq (f^{\prime})$ and let $g^{\prime \prime} = \frac{f}{f^{\prime \prime}}$. Since $f$ is a hyperplane arrangement we can choose $f^{\prime \prime}$ to be of any degree between $d^{\prime}$ and $d - 1$. Because $B_{f^{\prime}F}^{g} \subseteq B_{f^{\prime \prime}, F}^{g^{\prime \prime}}$, the containment \eqref{eqn - first approximate BS ideal containment} can be improved to
\[
B_{f^{\prime}F}^{g} \subseteq \mathbb{C}[S] \cdot \prod_{j=0}^{\mdr(f) + d - d^{\prime} - 3} \left( \sum_{k} d_{k} s_{k} + n + d^{\prime} + j \right).
\]
\end{proof}

\begin{remark} \label{remark-local to global b-poly}
\begin{enumerate}[(a)]
    \item It is easy to see, see Corollary 6 in \cite{BahloulOaku} for the $B_{F}$ statement, that
    \[
    B_{f^{\prime}F}^{g} = \bigcap\limits_{\x \in \mathbb{C}^{n}} B_{f^{\prime}F,\x}^{g}.
    \]
    \item Recall the notation of Definition \ref{def-hyperplane arrangement factorizations}. Given an edge $X \in L(A)$, there exists a $\x \in X$ such that $\x \notin \V(l_{m})$ for all $m \notin J(X)$. By definition, 
    \[
    F_{X} = (f_{X,1}, \dots, f_{X,r}) = ( \prod_{j \in S_{x,1}} l_{j}, \dots, \prod_{j \in S_{X,r}} l_{j}).
    \]
    We may write $F$ as 
    \[
    F= (\prod_{m \in S_{1} \setminus S_{X,1}}l_{m} \prod_{j \in S_{X,1}} l_{j} , \dots, \prod_{m \in S_{r} \setminus S_{X,r}}l_{m} \prod_{j \in S_{X,r}} l_{j} ).
    \]
    So at $\mathfrak{x}$, the decompositions $F$ and $F_{X}$ differ by multiplying each component by a unit at $\mathfrak{x}$. Arguing as in Lemma 10 of \cite{BahloulOaku} (see also the first paragraph of the proof of Theorem \ref{thm-unmixed symmetry}), we deduce
    \[
    B_{f^{\prime}F, \x}^{g} = B_{f_{X}^{\prime}F_{X}, \x}^{g_{X}}.
    \]
    Since $\x$ and $0$ both lie in the maximal edge of $f_{X}$, $B_{f_{X}^{\prime}F_{X},0}^{g_{X}} = B_{f_{X}^{\prime}F_{X},\x}^{g_{X}}.$
    The centrality of $f_{X}$, and the consequent $\mathbb{C}^{\star}$-action on $\V(f_{X})$, implies 
    \[
    B_{f_{X}^{\prime}F_{X},0}^{g_{X}} = B_{f_{X}^{\prime}F_{X}}^{g_{X}}.
    \]
    \item Putting (a) and (b) together yields
    \[
    B_{f^{\prime}F}^{g} = \bigcap_{X \in L(A)} B_{f_{X}^{\prime}F_{X}}^{g_{X}}.
    \]
\end{enumerate}
\end{remark}

The following definition will help simplify notation.

\begin{define} \label{def- simplifying polynomial}
Let $f=f_{1} \cdots f_{r}$ be any factorization of a central hyperplane arrangement and $F = (f_{1}, \dots, f_{r})$. Suppose $f^{\prime}$ divides $f$; $g = \frac{f}{f^{\prime}}$. For any indecomposable edge $X$ define the polynomial
\[
P_{f^{\prime}F,X}^{g} = \sum_{k} d_{X,k}s_{k} + r(X) + d_{X}^{\prime} \in \mathbb{C}[S].
\]

\end{define}

Remark \ref{remark-local to global b-poly} and Theorem \ref{thm-first subset thm} prove our version of Proposition 9 in \cite{MaisonobeFree}:

\begin{thm} \label{thm-main subset thm} \text{\normalfont (Compare to Proposition 9 of \cite{MaisonobeFree})}
Suppose $f$ is a central, tame, not necessarily reduced, hyperplane arrangement of rank $n$ and let $F=(f_{1}, \dots, f_{r})$ correspond to any factorization $f = f_{1} \cdots f_{r}$. Let $f^{\prime}$ divide $f$ and $g = \frac{f}{f^{\prime}}.$ For indecomposable edges $X$ of rank $\geq 2$ define

\[
p_{f^{\prime}F,X}(S) = \prod_{j_{X}=0}^{mdr(f_{X}) + d_{X} - d_{X}^{\prime} - 3} \left(P_{f^{\prime}F,X}^{g} + j_{X} \right).
\]
For indecomposable edges $X$ of rank one define
\[
p_{f^{\prime}F,X}(S) = \prod_{j_{X} = 0}^{d_{X} - d_{X}^{\prime} - 1} \left( P_{f^{\prime}F,X}^{g} + j_{X} \right).
\]
Then
\[
B_{f^{\prime}F}^{g} \subseteq \mathbb{C}[S] \cdot \lcm \left\{ p_{f^{\prime}F, X}(S) \mid X \in L(A), \ X \text{ \normalfont indecomposable} \right\}.
\]
\end{thm}
\begin{proof}
By Remark \ref{remark-local to global b-poly}, 
\[
B_{f^{\prime}F}^{g} = \left(\bigcap_{\substack{X \in L(A) \\ r(X) \geq 2}} B_{f_{X}^{\prime}F_{X}}^{g_{X}}\right) \bigcap \left(\bigcap_{\substack{X \in L(A) \\ r(X) = 1}} B_{f_{X}^{\prime}F_{X}}^{g_{X}}\right)
\]
If $X$ is an edge of rank $\geq 2$, then Theorem \ref{thm-first subset thm} combined with Definition \ref{def- simplifying polynomial} says 
\[
B_{f_{X}^{\prime}F_{X}}^{g_{X}} \subseteq \mathbb{C}[S] \cdot \prod_{j_{X} = 0}^{\mdr(f_{X}) + d_{X} - d_{x}^{\prime} - 3} (P_{f^{\prime}F,X}^{g} + j_{X}).
\]
Therefore, once we prove that for rank one edges $X$
\[
B_{f_{X}^{\prime}F_{X}}^{g_{X}} \subseteq \mathbb{C}[S] \cdot \prod_{j_{X} = 0}^{d_{X} - d_{X}^{\prime} - 1} \left( P_{f^{\prime}F,X}^{g} + j_{X} \right),
\]
then the claim will follow.

For the rank one edges, argue as in Theorem \ref{thm-first subset thm}. Since the rank is one, we can get an equation resembling \eqref{eqn-third eqn main thm star} without any $\psi_{f^{\prime}F}(\delta)$ terms and with $v = 1$. Now looking at the right constant terms, since $B(S) \in \mathbb{C}[S]$ and $L_{\textbf{0}}g$ is not, we deduce \eqref{eqn-divisibilty statement} holds with $\degree(v) = 0$. The other factors of $p_{f^{\prime}F}$ are found using the containment $B_{f^{\prime}F}^{g} \subseteq B_{f^{\prime \prime}F}^{g^{\prime \prime}}$, as in the final paragraph of Theorem \ref{thm-first subset thm}. 
\end{proof}

\subsection{An Element of $B_{f^{\prime}F}^{g}$} \text{ }

Here we drop the assumption of tameness and compute an element of $B_{f^{\prime}F}^{g}$ for $f = f_{1} \cdots f_{r}$ any factorization of a central, not necessarily reduced, hyperplane arrangement $f$ and where $f^{\prime}$ and $g$ are as before. The bulk of the argument tracks Proposition 10 of \cite{MaisonobeFree}, however we have removed the reducedness hypothesis. Again, we provide detail for the reader's sake.

We begin with some basic facts about differential operators. First, consider a product of functions $fg$ with factorizations $f = f_{1} \dots f_{r}$ and $g = g_{1} \dots, g_{u}$. Let $F = (f_{1}, \dots, f_{r})$ and $G = (g_{1}, \dots, g_{u})$ and $FG = (f_{1}, \dots, f_{r}, g_{1}, \dots, g_{u}).$ 

\begin{define} \label{def-d operator order and product}

Let $P \in \A[S]$ and consider $\A[S](FG)^{S}$. Relabel the $s_{k}$ so that we may write $\A[S, T]F^{S}G^{T} = \A[S] f_{1}^{s_{1}} \cdots f_{r}^{s_{r}} g_{1}^{t_{1}} \cdots g_{u}^{t_{u}}$ and consider $P$ as in $\A[S,T].$ As there is an $\A[S]$-action on $F^{S}$ there is a naturally defined $\A[S,T]$ action. Denote by $P \bullet F^{S}$ the result of letting $P$ act on $F^{S}$.
\end{define}

\begin{lemma} \label{lemma BS upper bound third lemma}
Let $P \in A_{n}(\mathbb{C})[S]$ of total order $k$, i.e. $P \in F_{(0,1,1)}^{k} A_{n}(\mathbb{C})[S]$. Then
\[
P F^{S} G^{T} - (P \bullet  F^{S}) G^{T} \in A_{n}(\mathbb{C})[S, T] F^{S} G^{T - k}.
\]
\end{lemma}
\begin{proof}
It is sufficient to prove the following: 

\emph{Claim}: If $h \in \mathbb{C}[X][S][T]$, there exists $Q_{\textbf{u}}$ of total order at most $\abs{\textbf{u}}$ such that
\[
\partial^{\textbf{u}} h  F^{S} G^{T} - h (\partial^{\textbf{u}} \bullet  F^{S}) G^{T} = Q_{\textbf{u}} F^{S} G^{T-\abs{\textbf{u}}}.
\]

We prove this by induction on $\abs{\textbf{u}}$. The base case is straightforward. For the inductive step, observe:
\begin{align} \label{eqn- inductive diff operator}
\partial_{1} \partial^{\textbf{u}} h  F^{S} G^{T} 
    &= \partial_{1}[ h (\partial^{\textbf{u}} \bullet F)G^{T} + Q_{\textbf{u}}  F^{S} G^{T - \abs{\textbf{u}}}] \\
    &= (\partial_{1} \bullet h)(\partial^{\textbf{u}} \bullet  F^{S}) G^{T} + h (\partial_{1} \partial^{\textbf{u}} \bullet F{^S}) G^{T} \nonumber \\
    &+ h(\partial^{\textbf{u}} \bullet F)(g \sum_{k} t_{k} \frac{\partial_{1} \bullet g_{k}}{g_{k}}) G^{T-1} + \partial_{1} Q_{\textbf{u}} F^{S}G^{T}. \nonumber
\end{align}
Since $\partial_{1} \bullet h \in \mathbb{C}[X][S][T]$ the induction hypothesis implies 
\[
(\partial_{1} \bullet h)(\partial^{\textbf{u}} \bullet  F^{S}) G^{T} \in F_{(0,1,1)}^{\abs{\textbf{u}}} \A[S][T] F^{S} G^{T - \abs{\textbf{u}}}.
\]
Similarly, since $h(g \sum_{k} t_{k} \frac{\partial_{1} \bullet g_{k}}{g_{k}}) \in \mathbb{C}[S][T]$, by induction
\[
h(\partial^{\textbf{u}} \bullet F^{S})(g \sum_{k} t_{k} \frac{\partial_{1} \bullet g_{k}}{g_{k}}) G^{T-1} \in F_{(0,1,1)}^{\abs{\textbf{u}}} \A[S][T] F^{S}G^{T - \abs{\textbf{u}}-1}.
\]
Rearranging \eqref{eqn- inductive diff operator} proves the claim and hence the lemma.
\end{proof}

We also need the following elementary lemma.

\begin{lemma} \label{lemma Euler transposition}
Let $E = x_{1} \partial_{1} + \cdots + x_{n} \partial_{n}$ be the Euler derivation. Then
\[
\prod_{j=0}^{t} (E + n + j) = \sum_{\substack{ u_{1}, \dots, u_{n} \\ u_{1} + \cdots + u_{n} = t+1}} \binom{t+1}{u_{1}, \cdots, u_{n}} \partial^{\textbf{u}}x^{\textbf{u}}.
\]

\end{lemma}

\begin{proof}
This also succumbs to induction on $t$ after utilizing Pascal's formula for multinomial coefficients.
\end{proof}

\begin{define}

Consider a central, essential, not necessarily reduced, hyperplane arrangement of rank $n$ defined by $f = l_{1} \cdots l_{d}$, where the $l_{k}$ are homogeneous linear forms. Write $L = (l_{1}, \dots, l_{d})$. For an edge $X \in L(A)$ and with $J(X)$ as in Definition \ref{def-hyperplane arrangement factorizations}, define the ideal $\Gamma_{L} \subseteq \mathbb{C}[x_{1}, \dots, x_{n}]$ by
\[
\Gamma_{L} = \sum_{\substack{ X \in L(A) \\ r(X) = n-1}} \mathbb{C}[x_{1}, \dots, x_{n}] \boldsymbol{\cdot} \prod_{k \notin J(X)} l_{k}.
\]

\end{define}

\begin{lemma} \label{lemma gamma radical}
Consider a central, essential, not necessarily reduced, hyperplane arrangement of rank $n$ defined by $f = l_{1} \cdots l_{d}$, where the $l_{k}$ are homogeneous linear forms. Let $L = (l_{1}, \dots , l_{d})$ and denote the ideal of $\mathbb{C}[x_{1}, \dots, x_{n}]$ generated by $x_{1}, \dots, x_{n}$ by $\mathfrak{m}$. Then there exists an integer $k$ such that $\mathfrak{m}^{k} \subseteq \Gamma_{L}.$
\end{lemma}

\begin{proof}
It suffices to show $\Gamma_{L}$ is $\mathfrak{m}$-primary since $\mathfrak{m}$ is maximal and $\mathbb{C}[x_{1}, \dots, x_{n}]$ is Noetherian. So we need only show $\V(\Gamma_{L}) = \{0\}.$ Suppose $0 \neq p \in \V(\Gamma_{L})$. Since $\V(\Gamma_{L})$ is the intersection of unions of central hyperplanes, we deduce $\V(\Gamma_{L})$ contains a codimension $n-1$ line. We may find a largest edge $X$ containing said line; if $X$ is not of codimension $n-1$ enlarge $X$ further to a codimension $n-1$ edge. So for all $k \notin J(X)$, $\V(l_{k})$ will not contain this line and hence will not contain $p$. But $p \in \V(\Gamma_{F}) \subseteq \V(\prod_{k \notin J(X)} l_{k}) = \cup_{k \notin J(X)} \V(l_{k})$, contradicting $p \in \V(\Gamma_{L})$. 
\end{proof}

\begin{remark}
We need essentiality in the above lemma lest the maximal edge of $L(A)$ have rank $n-1$ forcing $\Gamma_{F} = 1.$ Without this condition, the $X$ selected in the above proof could be the maximal edge of $L(A)$.
\end{remark}

Recall the notation of Definition \ref{def-hyperplane arrangement factorizations}. We proceed to the subsection's main idea, which is a generalization of Proposition 10 of \cite{MaisonobeFree} and is proved similarly.

\begin{thm} \label{thm- first supset thm} \text{\normalfont (Compare to Proposition 10 of \cite{MaisonobeFree})}
Consider a central, not necessarily reduced, hyperplane arrangement $f = l_{1} \cdots l_{d}$ where the $l_{k}$ are linear terms and let $L = (l_{1}, \dots, l_{d})$. Suppose that $f^{\prime}$ divides $f$; let $g = \frac{f}{f^{\prime}}.$ Then there is a positive integer $N$ such that
\[
\prod_{\substack{ X \in L(A) \\ X \text{ indecomposable }}} \prod_{j = 0}^{N} \left( P_{f^{\prime}L,X}^{g} + j \right) \in B_{f^{\prime}F}^{g}.
\]
\end{thm}

\begin{proof}

We prove this by induction on the rank of $L(A)$ and first deal with the inductive step. So we may assume the rank is $n$ and $f$ is essential. If $f$ is decomposable into $f_{1} f_{2}$, then $f^{\prime}$ (resp. $g$) inherts a decomposition $f_{1}^{\prime}f_{2}^{\prime}$ (resp. $g_{1}g_{2}$). If $F_{1}$ (resp. $F_{2}$) is the associated factorization of $f_{1}$ (resp. $f_{2}$) into linear forms and if $b_{1} \in B_{f_{1}^{\prime}F_{1}}^{g_{1}}$ and $b_{2} \in B_{f_{2}^{\prime}F_{2}}^{g_{2}}$, then $b_{1}b_{2} \in B_{f^{\prime}F}^{g}.$ In this case the induction hypothesis applies to $B_{f_{1}^{\prime}F_{1}}^{g_{1}}$ and $B_{f_{2}^{\prime}F_{2}}^{g_{2}}$. So we may assume $f$ is indecomposable. 

Let $\mathfrak{m}$ be the ideal in $\mathbb{C}[x_{1}, \dots, x_{n}]$ generated by $x_{1}, \dots, x_{n}$. On the one hand, Lemma \ref{lemma Euler transposition} implies that for all positive integers $t$
\[
\prod_{j=0}^{t} (s_{1} + \cdots + s_{d} + n + d^{\prime}+ j) f^{\prime}L^{S} = \prod_{j=0}^{t} (E + n + j) f^{\prime}L^{S} \in \A \cdot \mathfrak{m}^{t+1} f^{\prime}L^{S}.
\]
By Lemma \ref{lemma gamma radical}, for any positive integer $m$ there exists an integer $N$ large enough so that
\begin{equation} \label{eqn supset gamma eqn}
\prod_{j=0}^{N} (s_{1} + \cdots + s_{d} + n + d^{\prime}+ j) f^{\prime}L^{S} \in \sum_{\substack{ X \in L(A) \\ r(X) = n-1}} \A[S] (\prod_{k \notin J(X)} l_{k})^{m} f_{X}^{\prime}L^{S}.
\end{equation} 
Note we have folded some of the factors of $f^{\prime}$ into $(\prod_{k \notin J(X)} l_{k})^{m}$.

By induction, for each such edge $X$ of rank less than $n$, there exists a differential operator $P_{X}$ of total order $k_{X}$ and a polynomial $b_{X} \in \mathbb{C}[S]$ such that $P_{X} \prod_{i \in J(X)} l_{i}^{s_{i}+1} = b_{X} f_{X}^{\prime} \prod_{i \in J(X)} l_{i}^{s_{i}}.$ Fix $m$ large enough so that $m > \text{max} \{k_{X} \mid X \in L(A), \ X \text{ codimension } n-1 \}$. Consequently, choose $N$ large enough so that (\ref{eqn supset gamma eqn}) holds for this fixed $m$. Lemma \ref{lemma BS upper bound third lemma} implies 
\begin{align} \label{eqn supset induction b-poly}
b_{X} (\prod_{k \notin J(X)} l_{k})^{m} f_{X}^{\prime}L^{S} & = (b_{X} f_{X}^{\prime} \prod_{i \in J(X)} l_{i}^{s_{i}}) (\prod_{k \notin J(X)} l_{k}^{s_{k} + m})  \\
    & \in \A[S] (\prod_{i \in J(X)} l_{i}^{s_{i} + 1})(\prod_{k \notin J(X)} l_{k}^{s_{k} + m - k_{X}}) \nonumber \\
    & \subseteq \A [S] L^{S+1}. \nonumber
\end{align}

Combining (\ref{eqn supset gamma eqn}) and (\ref{eqn supset induction b-poly}) we deduce 
\begin{equation} \label{eqn supset inductive step completed}
\prod_{j=0}^{N} (s_{1} + \cdots + s_{d} + n + d^{\prime} + j) (\prod_{\substack{ X \in L(A) \\  r(X)=  n-1}} b_{X}) f^{\prime} L^{S} \in \A [S] L^{S+1}.
\end{equation}
The result follows by the inductive description of each $b_{X}$ and the definition of $P_{f^{\prime}L,X}^{g}.$ Note we may have to replace either the $N$ chosen in \eqref{eqn supset inductive step completed} or the $N$ coming from the inductive hypothesis with a larger integer so that the final polynomial is in the promised form. There is no harm in this as it can only only add linear factors to the polynomial appearing in (\ref{eqn supset inductive step completed}) and does not change the containment.

All that remains is the base case, but this is obvious by a direct computation using Lemma \ref{lemma Euler transposition}.
\end{proof}

This theorem only gives an element of $B_{f^{\prime}L}^{g}$ when $L$ is a factorization into linear forms. If $f$ is tame we can find an element no matter the factorization.

\begin{cor} \label{cor-main BS member any factorization}
Let $f = f_{1} \cdots f_{r}$ be a central, not necessarily reduced, tame hyperplane arrangement where the $f_{k}$ are not necessarily linear forms. Let $F = (f_{1}, \dots, f_{r})$. Suppose $f^{\prime}$ divides $f$; let $g = \frac{f}{f^{\prime}}.$ If $L$ corresponds to the factorization of $f$ into linear terms, then there exists a positive integer $N$ such that
\[
\prod_{\substack{ X \in L(A) \\ X \text{ indecomposable }}} \prod_{j = 0}^{N} \left(P_{f^{\prime}L,X}^{g} + j \right) \text{\normalfont modulo } S_{F} \in B_{f^{\prime}F}^{g},
\]
where $S_{F}$ is as in Definition \ref{def-compatible with G}.
\end{cor}

\begin{proof}
Use Proposition \ref{prop-BS ideal subset, coarser}.
\end{proof}

Just as in the last part of Theorem 2 of \cite{MaisonobeFree}, \ref{thm- first supset thm} also implies $B_{f^{\prime}L}^{g}$ is principal. (Here we very much need $L$ to correspond to a factorization into linear forms.)

\begin{cor} \label{cor- b-ideal linear forms principal}
Consider the central, not necessarily reduced, free hyperplane arrangement $f = l_{1} \cdots l_{d}$, where the $l_{k}$ are linear forms, and let $L = (l_{1}, \cdots, l_{d})$. Suppose $f^{\prime}$ divides $f$; let $0 \neq g$ divide $\frac{f}{f^{\prime}}$. Then $B_{f^{\prime}L}^{g}$ equals its radical and is principal. 
\end{cor}

\begin{proof}
Let $P(S)$ be the polynomial of Theorem \ref{thm- first supset thm}. If $g$ divides $\frac{f}{f^{\prime}}$, then by said theorem $P(S) \in B_{f^{\prime}L}^{g}$. The claim then follows by Proposition \ref{prop- principal radical} and Proposition \ref{prop- criterion for principal b-ideal} since $P(S)$ cuts out a reduced hyperplane arrangement.
\end{proof}

\subsection{Computations and Estimates} \text{ }

We now have combinatorial determined ideal subsets and supsets of $B_{f^{\prime}F}^{g}$. In general, $\V(B_{f})$ is not combinatorially determined. However, if $f$ is tame, then $\V(B_{f}) \cap [-1,0]$ is combinatorial.

\begin{thm} \label{thm-combinatorial roots}
Let $f$ be a central, not necessarily reduced, tame hyperplane arrangement. Suppose $f^{\prime}$ divides $f$; let $g = \frac{f}{f^{\prime}}.$ Then the roots $\V(B_{f^{\prime}f}^{g})$ lying in $[-1,0)$ are combinatorially determined:
\[
\V(B_{f^{\prime}f}^{g}) \cap [-1,0) = \bigcup_{\substack{ X \in L(A) \\ X \text{ indecomposable}}} \bigcup_{j_{X}= r(X) + d_{X}^{\prime}}^{d_{X}}  \frac{ - j_{X}}{d_{X}}.
\]
Setting $f^{\prime} = 1$ gives the roots of the Bernstein--Sato polynomial of $f$ lying in $[-1,0).$
\end{thm}
\begin{proof}

We find a subset and supset of $B_{f^{\prime}F}^{g}$ using Corollary \ref{cor-main BS member any factorization} and Theorem \ref{thm-main subset thm} respectively. Their varieties will be equal after intersecting with $[-1,0)$ once we verify the following inequalities for indecomposable edges $X$: $r(X) + \mdr(f) + d_{X} - 3 \geq d_{X}$ if $r(X) \geq 2$; $1 + d_{X} - 1 \geq d_{X}$ if $r(X) =1.$ The second is trivial. The first is as well: since $X$ is indecomposable $\mdr(f) \geq 2$. 
\end{proof}

\begin{example} \label{ex- roots not combinatorial}
In \cite{uli}, Walther showed the Bernstein--Sato polynomial of an arrangement is not combinatorially determined. He gives the following two arrangements that have the same intersection lattice, but the former has $\frac{-18 + 2}{9}$ as a root and the latter does not:
\begin{align*}
f = xyz(x + 3z)(x + y + z)(x + 2y + 3z)(2x + y + z)(2x + 3y + z)(2x + 3y + 4z); \\
g = xyz(x + 5z)(x + y + z)(x + 3y + 5z)(2x + y + z)(2x + 3y + z)(2x + 3y + 4z).
\end{align*}
Because these arrangements are rank $3$ they are automatically tame, cf. Remark \ref{rmk- logarithmic forms remark}. The above theorem says the roots of the b-polynomials agree inside $[-1,0)$. In Remark 4.14.(iv) of \cite{SaitoArrangements}, Saito shows that their roots agree except for $\frac{-18 + 2}{9}.$
\end{example}

For the rest of the subsection we restrict to free hyperplane arrangments. In \cite{MaisonobeFree}, Maisonobe used the symmetry of $B_{L}$, when $L$ corresponded to a factorization of a reduced $f$ into linear terms, to make his estimates of $B_{L}$ so precise they actually computed $B_{L}$, cf. Theorem 2 in loc. cit. We use the symmetry of $B_{f^{\prime}F}^{g}$ given by $\varphi$ of Theorem \ref{thm-unmixed symmetry} similarly, but our situation is more technical because of the addition of $f^{\prime}$, the lack of reducedness, and our focus on different factorizations $F$. 

\begin{lemma} \label{lemma - varphi simple form}
Let $f= f_{1} \cdots f_{r}$ be an unmixed factorization of a central hyperplane arrangement and let $F=(f_{1}, \dots, f_{r})$. Suppose $f^{\prime}$ divides $f$; $g = \frac{f}{f^{\prime}}.$ If $(f^{\prime}, F)$ is an unmixed pair and $\varphi$ the $\mathbb{C}[S]$-automorphism prescribed in Theorem \ref{thm-unmixed symmetry}, then
\[
\varphi(P_{f^{\prime}F,X}^{g}) = -( P_{f^{\prime}F,X}^{g} + d_{X,\red} + d_{X} - 2r(X) - d_{X}^{\prime}).
\]
\end{lemma}
\begin{proof}
First notation. Factor $f = l_{1}^{v_{1}} \cdots l_{q}^{v_{q}}$, where the $l_{t}$ pairwise distinct irreducibles. Let $\{m_{k}\}$ be the repeated multiplicities of $F$; $\{d_{k}^{\prime}, d_{k} \}_{k}$ and $\{d_{k}^{\prime \prime}, d_{k}\}_{k}$ the repeated powers of the unmixed pairs $(f^{\prime}, F)$ and $(g, F).$ Because $f^{\prime} g = f$, the formulation of $\varphi$ in Theorem \ref{thm-unmixed symmetry} can be simplified: 
\begin{align} \label{eqn- varphi lemma first simplification}
\varphi(\sum_{k} d_{X,k} s_{k}) 
    &= - \sum_{k} d_{X,k} (s_{k} + \frac{1}{m_{k}} + \frac{2 d_{k}^{\prime}}{d_{k}} + \frac{d_{k}^{\prime \prime}}{d_{k}}) \nonumber \\
    &= - \sum_{k} d_{X,k} (s_{k} + \frac{1}{m_{k}} + \frac{ d_{k}^{\prime}}{d_{k}} + 1) \nonumber \\
    &= - \sum_{k} d_{X,k} (s_{k} + \frac{1}{m_{k}}) - \sum_{k} d_{X,k,\red} d_{k}^{\prime} - d_{X} \nonumber \\
    &= - \sum_{k} d_{X,k} (s_{k} + \frac{1}{m_{k}}) - d_{X}^{\prime} - d_{X}. \nonumber
\end{align}
After rearranging, we will be done once we show that $ \sum_{k} \frac{d_{X,k}}{m_{k}} = d_{X,red}$.

Fix $k \in [r]$. Observe:
\begin{equation} \label{eqn- varphi lemma mk 1}
\prod_{\substack{t \in [q] \\  v_{t} = m_{k} }} l_{t}^{m_{k}} = \prod_{\substack{ i \in [r] \\ m_{i} = m_{k} }} f_{i} = \prod_{\substack{ i \in [r] \\ m_{i} = m_{k} }} \prod_{\substack{ t \in [q] \\ f_{i} \in (l_{t}) }} l_{t}^{d_{i}}.
\end{equation}
Equality will still hold in \eqref{eqn- varphi lemma mk 1} if we further restrict $t$ to the integers such that $l_{t}$ divides $f_{X}$. The degrees of the resulting polynomials are equal:
\begin{align} \label{eqn- varphi lemma mk 2}
m_{k} \ \abs{ \{l_{t} \mid v_{t} = m_{k}; f_{X} \in (l_{t}) \}} 
    &= \sum_{\substack{ i \in [r] \\ m_{i} = m_{k}}} d_{i} \ \abs{ \{l_{t} \mid f_{i}, f_{X} \in (l_{t}) \}} \\
    &= \sum_{\substack{ i \in [r] \\ m_{i} = m_{k}}} d_{i} \ d_{X,i,\red} \nonumber \\
    &= \sum_{\substack{ i \in [r] \\ m_{i} = m_{k}}} d_{X,i}. \nonumber
\end{align}
Therefore 
\begin{align} \label{eqn- varphi lemma mk 3}
\sum_{k} \frac{d_{X,k}}{m_{k}} 
    = \sum_{p \in \{ m_{k} \}} \sum_{\substack{ i \in [r] \\ m_{i} = p}} \frac{d_{X,k}}{p} 
    &= \sum_{p \in \{ m_{k} \}} \abs{ \{l_{t} \mid v_{t} = p; f_{X} \in (l_{t}) \}} \\ 
    &= \sum_{p \in \{ v_{t} \}} \abs{ \{l_{t} \mid v_{t} = p; f_{X} \in (l_{t}) \}} \nonumber \\ 
    &= d_{X,\red}. \nonumber
\end{align}
\end{proof}

First we use Theorem \ref{thm- first supset thm} and the symmetry of $B_{f^{\prime}L}^{g}$ to find an element of $B_{f^{\prime}L}^{g}$ that more accurately approximates the Bernstein--Sato ideal.

\begin{prop} \label{prop- linear forms, symmetry element of b-ideal}
Consider the central, not necessarily reduced, free hyperplane arrangement $f = l_{1} \cdots l_{d}$, where the $l_{k}$ are linear forms, and let $L = (l_{1}, \dots, l_{d})$. Suppose $f^{\prime}$ divides $f$; let $g = \frac{f}{f^{\prime}}$. Then 
\begin{equation} \label{eqn- symmetry element linear form factorization}
\prod_{\substack{X \in L(A) \\ X \text{ indecomposable }}} \prod_{j_{X} = 0}^{d_{x,\red} + d_{X} - 2r(X) - d_{X}^{\prime}} \left( P_{f^{\prime}L,X}^{g}+ j_{X} \right) \in B_{f^{\prime}L}^{g}.
\end{equation}
\end{prop}

\begin{proof}
By Theorem \ref{thm- first supset thm} there exists a positive integer $N$ such that
\begin{equation} \label{eqn- 1 linear forms, symmetry element}
\prod_{\substack{X \in L(A) \\ X \text{ indecomposable }}} \prod_{j_{X} = 0}^{N} \left( P_{f^{\prime}L,X}^{g}+ j_{X} \right) \in B_{f^{\prime}L}^{g}.
\end{equation}
Since $(f^{\prime}, L)$ are an unmixed pair up to units by virtue of $L$ being a factorization into linear forms, by Theorem \ref{thm-unmixed symmetry}/Corollary \ref{cor- symmetry formula list} and Lemma \ref{lemma - varphi simple form}
\begin{equation} \label{eqn-2 linear forms, symmetry element}
\prod_{\substack{X \in L(A) \\ X \text{ indecomposable }}} \prod_{j_{X} = 0}^{N} \left( P_{f^{\prime}L,X}^{g}+ d_{X,\red} + d_{X} - 2r(X) - d_{X}^{\prime} - j_{X} \right) \in B_{f^{\prime}L}^{g}.
\end{equation}
By Corollary \ref{cor- b-ideal linear forms principal}, $B_{f^{\prime}L}^{g}$ is principal. Comparing the irreducible factors of the elements given in \eqref{eqn- 1 linear forms, symmetry element} and \eqref{eqn-2 linear forms, symmetry element} proves the claim.
\end{proof}

When the rank of $f$ is at most $2$, and so $f$ is automatically free, we can compute $\V(B_{f^{\prime}F}^{g})$ for any factorization $F$ of $f$ and we can compute $B_{f^{\prime}L}^{g}$ for $L$ a factorization into linear terms. 

\begin{thm} \label{thm- rank at most 2 computation}
Suppose that $f$ is a central, not necessarily reduced, hyperplane arrangement of rank at most $2$ and let $F = (f_{1}, \dots, f_{r})$ correspond to any factorization $f = f_{1} \cdots f_{r}$. Let $f^{\prime}$ divide $f$ and $g = \frac{f}{f^{\prime}}.$ Then
\begin{equation} \label{eqn-statement 1 rank at most 2}
\V( B_{f^{\prime}F}^{g}) = \V \left( \prod_{\substack{X \in L(A) \\ X \text{ indecomposable }}} \prod_{j_{X} = 0}^{d_{X,\red} + d_{X} - 2r(X) - d_{X}^{\prime}} \left( P_{f^{\prime}F,X}^{g} + j_{X} \right) \right).
\end{equation}
If $L$ is a factorization of $f = l_{1} \cdots l_{d}$ into irreducibles, then
\begin{equation} \label{eqn-statement 2 rank at most 2}
 B_{f^{\prime}L}^{g} =  \prod_{\substack{X \in L(A) \\ X \text{ indecomposable }}} \prod_{j_{X} = 0}^{d_{X,\red} + d_{X} - 2r(X) - d_{X}^{\prime}} \left( P_{f^{\prime}L,X}^{g} + j_{X} \right).
\end{equation}
\end{thm}

\begin{proof}
If $f$ is indecomposable, then by Saito's criterion for freeness, cf. page 270 of \cite{SaitoLogarithmic}, $\mdr(f) = d_{\red} - 1.$ So in this case Theorem \ref{thm-main subset thm} implies
\begin{equation} \label{eqn- rank at most 2 contained in rad}
B_{f^{\prime}F}^{g} \subseteq \sqrt{ \mathbb{C}[S] \cdot  \prod\limits_{j_{0} = 0}^{d_{\red} + d - d^{\prime} - 4} \left(P_{f^{\prime}F,0}^{g} + j_{0} \right) \prod\limits_{\substack{X \in L(A) \\ r(X) = 1}} \prod_{j_{X}=0}^{d_{X} - d_{X}^{\prime} - 1} \left( P_{f^{\prime}F,X}^{g} + j_{X} \right) }.
\end{equation} 
Proposition \ref{prop- linear forms, symmetry element of b-ideal} and Proposition \ref{prop-BS ideal subset, coarser} together imply 
\begin{equation} \label{eqn-rank at most 2 contains symmetric coarser}
\sqrt{ \mathbb{C}[S] \cdot \prod_{\substack{X \in L(A) \\ X \text{ indecomposable }}} \prod_{j_{X} = 0}^{d_{x,\red} + d_{X} - 2r(X) - d_{X}^{\prime}} \left( P_{f^{\prime}F,X}^{g}+ j_{X} \right)} \subseteq \sqrt{ B_{f^{\prime}F}^{g}},
\end{equation}
where we have included radicals because the image of a polynomial modulo $S_{F}$ may have multiplicands with large multiplicities, cf. Example \ref{ex- going modulo}. Combining \eqref{eqn- rank at most 2 contained in rad} and \eqref{eqn-rank at most 2 contains symmetric coarser} and simplifying $d_{x,\red} + d_{X} - 2r(X) - d_{X}^{\prime}$ for rank $2$ and rank $1$ edges proves \eqref{eqn-statement 1 rank at most 2}. 

Because $L$ is a factorization into irreducibles, even if $f$ is not reduced the polynomial on the right hand side of \eqref{eqn-statement 2 rank at most 2} is reduced. Therefore \eqref{eqn-statement 1 rank at most 2} and Corollary \ref{cor- b-ideal linear forms principal} implies \eqref{eqn-statement 2 rank at most 2}. The case of $f$ decomposable follows by similar reasoning.
\end{proof}

If $f$ is of rank greater than $2$, $\mdr(f)$ can be small and so the estimate in Theorem \ref{thm-main subset thm} will not be precise enough for our purposes. In this case, we impose symmetry on $B_{f^{\prime}F}^{g}$ to obtain the following estimates:

\begin{thm} \label{thm - final estimation}
Suppose that $f = f_{1} \cdots f_{r}$ is a central, not necessarily reduced, free hyperplane arrangement, $F = (f_{1}, \cdots, f_{r})$, $f^{\prime}$ divides $f$, and $g = \frac{f}{f^{\prime}}.$ Then 
\begin{equation} \label{eqn-final estimation theorem statement 1}
\sqrt{\mathbb{C}[S] \cdot \prod_{\substack{X \in L(A) \\ X \text{ indecomposable }}} \prod_{j_{X} = 0}^{d_{X,\red} + d_{X} - 2r(X) - d_{X}^{\prime}} \left( P_{f^{\prime}F,X}^{g} + j_{X} \right)} \subseteq  \sqrt{B_{f^{\prime}F}^{g}}.
\end{equation}
If we assume $(f^{\prime}, F)$ is an unmixed pair up to units, then 
\begin{equation} \label{eqn-final estimation theorem statement 2}
B_{f^{\prime}F}^{g} \subseteq \sqrt{\mathbb{C}[S] \cdot \prod_{\substack{X \in L(A) \\ X \text{ indecomposable }}} \prod_{j_{X} \in \Xi_{X}} \left( P_{f^{\prime}F,X}^{g} + j_{X} \right)},
\end{equation}
where, for each indecomposable edge $X$, $\Xi_{X}$ is the, possibly empty, set of nonnegative integers defined by
\[
\begin{cases}
    [0, d_{X,\red} + d_{X} - 2r(X) - d_{X}^{\prime}] & r(X) \leq 2 \\
    [0, d_{X} - d_{X}^{\prime} - 1] \cup [d_{X,\red} - 2r(X) + 1, d_{X,\red} + d_{X} - 2r(X) - d_{X}^{\prime}] & r(X) \geq 3.
    \end{cases}
\]
\end{thm}
\begin{proof}
The inclusion \eqref{eqn-final estimation theorem statement 1} is proved in exactly the same way as \eqref{eqn-rank at most 2 contains symmetric coarser}, so we need to only prove \eqref{eqn-final estimation theorem statement 2}. Arguing as in the beginning of Theorem \ref{thm-unmixed symmetry}, we may assume $(f^{\prime},F)$ is an unmixed pair. Theorem \ref{thm-main subset thm} implies
\begin{align} \label{eqn- final estimation thm eq 1}
B_{f^{\prime}F}^{g} \subseteq \sqrt{\mathbb{C}[S] \cdot \prod_{\substack{X \in L(A) \\ X \text{ indecomposable} \\ r(X) \geq 3}} \prod_{j_{X} = 0}^{d_{X} - d_{X}^{\prime} - 1} \left( P_{f^{\prime}F,X}^{g}+ j_{X} \right)}.
\end{align}
The symmetry of $B_{f^{\prime}F,X}^{g}$, cf. Theorem \ref{thm-unmixed symmetry}/Corollary \ref{cor- symmetry formula list}, Lemma \ref{lemma - varphi simple form}, and \eqref{eqn- final estimation thm eq 1} imply
\begin{align} \label{eqn-final estimation thm eq 2}
B_{f^{\prime}F}^{g}
    & \subseteq \sqrt{\mathbb{C}[S] \cdot \prod_{\substack{X \in L(A) \\ X \text{ indecomposable } \\ r(X) \geq 3}} \prod_{j_{X} = 0}^{d_{X} - d_{X}^{\prime} - 1}  P_{f^{\prime}F,X}^{g}+ d_{X, \red} + d_{X} - 2r(X) - d_{X}^{\prime} - j_{X} }  \\
    &= \sqrt{ \mathbb{C}[S] \cdot \prod_{\substack{X \in L(A) \\ X \text{ indecomposable } \\ r(X) \geq 3}} \prod_{j_{X} = d_{x,\red} - 2r(X) + 1}^{d_{x,\red} + d_{X} - 2r(X) - d_{X}^{\prime}} \left( P_{f^{\prime}F,X}^{g}+ j_{X} \right) }. \nonumber
\end{align}
At the edges of rank two or one we have an ideal containment similar to \eqref{eqn- rank at most 2 contained in rad}. Combining this, \eqref{eqn- final estimation thm eq 1}, and \eqref{eqn-final estimation thm eq 2} and using the fact that $\mathbb{C}[S]$ is a UFD proves \eqref{eqn-final estimation theorem statement 2}.
\end{proof}

If $d^{\prime}$ is small enough, the previous result does not just estimate--it computes.

\begin{cor} \label{cor-final estimation cor} \text{\normalfont (Compare to Theorem 2 of \cite{MaisonobeFree})} 
Suppose $f = f_{1} \cdots f_{r}$ is a central, not necessarily reduced, free hyperplane arrangement, $F = (f_{1}, \cdots, f_{r})$, $f^{\prime}$ divides $f$, and $g = \frac{f}{f^{\prime}}.$ If $(f^{\prime}, F)$ is an unmixed pair up to units and if $d^{\prime} \leq 4$, then 
\begin{equation} \label{eqn-final estimation cor statement 1}
\V( B_{f^{\prime}F}^{g}) = \V \left( \prod_{\substack{X \in L(A) \\ X \text{ indecomposable }}} \prod_{j_{X} = 0}^{d_{X,\red} + d_{X} - 2r(X) - d_{X}^{\prime}} \left( P_{f^{\prime}F,X}^{g} + j_{X} \right) \right).
\end{equation}
If $L$ is a factorization of $f = l_{1} \cdots l_{d}$ into irreducibles and $d^{\prime} \leq 4$, then
\begin{equation} \label{eqn-final estimation cor statement 2}
 B_{f^{\prime}L}^{g} =  \prod_{\substack{X \in L(A) \\ X \text{ indecomposable }}} \prod_{j_{X} = 0}^{d_{X,\red} + d_{X} - 2r(X) - d_{X}^{\prime}} \left( P_{f^{\prime}L,X}^{g} + j_{X} \right).
\end{equation}
If $f^{\prime} = 1$ and $f$ is reduced, then for any $F$
\begin{equation} \label{eqn-final estimation cor statement 3}
\V( B_{F}) = \V \left( \prod_{\substack{X \in L(A) \\ X \text{ indecomposable }}} \prod_{j_{X} = 0}^{d_{X,\red} + d_{X} - 2r(X)} \left( P_{F,X}^{g} + j_{X} \right) \right).
\end{equation}
In particular, if $f$ is reduced or is a power of a central, reduced, and free hyperplane arrangement, then the roots of the Bernstein--Sato polynomial of $f$ are given by \eqref{eqn-final estimation cor statement 3}. 
\end{cor}

\begin{proof}
Because of Theorem \ref{thm - final estimation}, proving \eqref{eqn-final estimation cor statement 1} amounts to showing that $\Xi_{X} = [0, d_{X,\red} + d_{X} - 2r(X) - d_{X}^{\prime}]$ for each $X$ of rank at least $3$. This occurs if $d_{X}^{\prime} \leq 2(r(X) - 1)$. So \eqref{eqn-final estimation cor statement 1} is true. Since $(f^{\prime},L)$ is always an unmixed pair up to units, Corollary \ref{cor- b-ideal linear forms principal} proves \eqref{eqn-final estimation cor statement 2}. Equation \eqref{eqn-final estimation cor statement 3} follows from \eqref{eqn-final estimation cor statement 1} and the fact $(1, F)$ is always an unmixed pair up to units when $f$ is reduced, cf. Corollary \ref{cor- symmetry formula list}. For the final claim, it suffices to note that $(1, F)$ for $F = (f)$ is an unmixed pair up to units provided $f$ is reduced or $f$ is a power of a central, reduced hyperplane arrangement. 
\end{proof}

\begin{remark} \label{rmk- gesturing at Budur improvement}
\begin{enumerate} [(a)]
    \item Let us outline how to strengthen the final claim of Corollary \ref{cor-final estimation cor} to Bernstein--Sato polynomials for all non-reduced, free $f$. In the recently announced paper \cite{BudurConjecture},  Budur, Veer, Wu, and Zhou consider local, analytic $f$ that satisfy a vanishing $\text{Ext}$ criterion. Namely, that $\text{Ext}_{\mathscr{D}_{X,\x}[S]}^{k}(\mathscr{D}_{X,\x}[S]F^{S}, \mathscr{D}_{X,\x}[S])$ vanishes for all but one value of $k$. (We let $F$ corresponds to any factorization of $f$.) In Proposition 3.4.3 they characterize elements of $\V(B_{F,\x})$ in terms of the non-vanishing of a certain tensor product. It is easy to show that this is equivalent to the non-surjectivity of the $\mathscr{D}_{X,\x}$-map $\nabla_{A}$. This is the map $\mathscr{D}_{X,\x}[S]F^{S} / (s_{1}-a_{1}, \dots, s_{r}-a_{r}) \cdot \mathscr{D}_{X,\x}[S]F^{S}  \to \mathscr{D}_{X,\x}[S]F^{S} / (s_{1}-(a_{1}-1), \dots, s_{r}-(a_{r}-1)) \cdot \mathscr{D}_{X,\x}[S]F^{S}$ induced by sending each $s_{k}$ to $s_{k+1}$. Here $A$ corresponds to $(a_{1}, \dots, a_{r}) \in \mathbb{C}^{r}$. See Section 3 of \cite{me}, Proposition 2 of \cite{BudurLocalSystems}, or Appendix B in this paper for more details on $\nabla_{A}$. 
    If $f$ corresponds to a free, possibly non-reduced, arrangement, it follows from Theorem \ref{thm- useful duality formula} that the vanishing $\text{Ext}$ condition of \cite{BudurConjecture} holds. Moreover, using the commutative diagram in Remark 3.3 of \cite{me}, the non-surjectivity of the map $\nabla_{A}$ is equivalent to the non-surjectivity of the classical map $\nabla_{a}$. (This is the same as $\nabla_{A}$ for $r=1$.) The non-surjectivity of $\nabla_{a}$ is known to characterize the roots of the Bernstein--Sato polynomial of an arbitrary 
    $f$. So when $L$ corresponds to a factorization of our possibly non-reduced arrangement $f$ into irreducibles, we can use the above procedure to show that intersecting $\V(B_{L})$ with the diagonal gives $\V(B_{f})$, again, see Remark 3.3 of \cite{me}. Using the formula for $\V(B_{L})$ in \eqref{eqn-final estimation cor statement 2}, we then obtain the expected formula \eqref{eqn-final estimation cor statement 3} for $\V(B_{f})$ without requiring the reduced hypothesis.
    \item The above strategy for computing $\V(B_{f})$ for $f$ a central, reduced, free hyperplane arrangement can also be executed without appeal to \cite{BudurConjecture} thanks to Proposition \ref{prop - appendix b prop}.
    \item In light of Proposition 3.4.3 of \cite{BudurConjecture}, the assumption of ``unmixed pair up to units" does not seem to be necessary. Rather, it seems there should be a version of this result for $f^{\prime}F^{S}$ so that computing $B_{f^{\prime}L}^{g}$ would be sufficient for computing $\V(B_{f^{\prime}F}^{g})$.
\end{enumerate}
\end{remark}

\section{Freeing Hyperplane Arrangements}

In this short section we consider the problem of embedding a central hyperplane arrangement $g$ inside a central, free hyperplane arrangement. Equivalently, given such a $g$ we consider central hyperplane arrangements $f$ such that $fg$ is free. (Note that we have somewhat switched notation for reasons that will become clear in Proposition \ref{prop- freeing arrangements, embedding}.)

\begin{define}
We say the central arrangement $f$ \emph{frees} the central arrangement $g$ if $fg$ is free. 
\end{define}

For $g$ an arbitrary divisor, it is unknown if such an $f$ exists. In \cite{MondSchulzeAdjoint}, Mond and Schulze find some general instances of the freeing divisor $f$; see also \cite{newfromold}, \cite{damon}, \cite{simis}. Returning to arrangements $g$, both Abe and Wakefield identify some situations in \cite{Abe} and \cite{Wakefield} respectively where $f$ is a hyperplane and $fg$ is free. For $g$ a central hyperplane arrangement, Masahiko Yoshinaga \cite{Yoshi} has communicated to us an algorithm, depending only on the intersection lattice of $g$, that always produces such an $f$. Accordingly, we make the following definition, noting nothing is lost by assuming reducedness.

\begin{define}
For $g$ a central, reduced hyperplane arrangement, define
\[
\mu_{g} = \min \{ \deg(f) \mid f \text{ is a central arrangement that frees } g \}.
\]
\end{define}

We will highlight a connection between small roots of the Bernstein--Sato polynomial of a tame $g$ and lower bounds for $\mu_{g}$. First some notation. 

Consider a reduced hyperplane arrangement $l_{1} \cdots l_{d}$ and write it as a product $fg$. Let $F= (f_{1}, \dots, f_{r})$ and $G = (g_{1}, \dots, g_{u})$ correspond to the factorizations $f = f_{1} \cdots f_{r}$ and $g=g_{1} \cdots g_{u}$ into linear terms and let $FG$ correspond to the factorization $l_{1} \cdots l_{d} = f_{1} \cdots f_{r} \cdot g_{1} \cdots g_{u}.$ When considering the $\A[S]$-module generated $(FG)^{S}$, we will re-label so this is an $\A[S,T]$-module generated by $f_{1}^{s_{1}} \cdot f_{r}^{s_{r}} g_{1}^{t_{1}} \cdot g_{u}^{t_{u}}$. Finally, let $S+1$ denote the $\mathbb{C}[S]$ ideal generated by $s_{1} + 1, \dots, s_{r} + 1$ and let $\Delta_{S-1}: \mathbb{C}^{u} \to \mathbb{C}^{r+u} = \mathbb{C}^{d}$ be the embedding given by $(a_{1}, \dots, a_{u}) \mapsto (-1, \dots, -1, a_{1}, \dots, a_{u}).$

We need the following result:

\begin{prop} \label{prop- freeing arrangements, embedding} Let $f, g, F, G$ be as in the preceding paragraph. Suppose $fg$ is tame. Then
\[
\Delta_{S-1}(\V(B_{G})) \subseteq \V(B_{f FG}^{g}) \cap \{s_{1} = -1, \dots, s_{r} = -1\} \subseteq \mathbb{C}^{u+r}.
\]
\end{prop} 

\begin{proof}
Define $I = \A[S,T] \cdot \ann_{\A[T]}G^{T} + \A[S,T] \cdot g + \A[S,T] \cdot (S+1).$ If $P \in I \cap \mathbb{C}[S,T]$, then 
\[
P \text{ modulo } \A[S,T] \cdot (S+1) \in \mathbb{C}[T] \cap (\A[T] \cdot \ann_{\A[T]}G^{T} + \A[T] \cdot g).
\] So 
\[
I \cap \mathbb{C}[S,T] \subseteq \mathbb{C}[S,T] \cdot B_{G} + \mathbb{C}[S,T] \cdot (S+1).
\]
As the reverse equality is obvious, 
\[
I \cap \mathbb{C}[S,T] = \mathbb{C}[S,T] \cdot B_{G} + \mathbb{C}[S,T] \cdot (S+1).
\]

For $\delta$ a logarithmic derivation of $fg$, 
\[
\psi_{fFG}(\delta) = \delta - \sum_{k} s_{k} \frac{\delta \bullet f_{k}}{f_{k}} - \sum_{m} t_{m} \frac{\delta \bullet g_{m}}{g_{m}} - \frac{\delta \bullet f}{f}. 
\]
Under the map $\A[S,T] \mapsto \A[S,T] / \A[S,T] \cdot (S+1)$, 
\[
\psi_{fFG}(\delta) \mapsto \delta - \sum_{m} t_{m} \frac{\delta \bullet g_{m}}{g_{m}} = \psi_{G}(\delta) \in \ann_{\A[T]}G^{T}.
\]
Therefore 
\[
I \supseteq \A[S,T] \cdot \theta_{fFG} + \A[S,T] \cdot g + \A[S,T] \cdot (S+1).
\]
Intersecting with $\mathbb{C}[S,T]$ and using Corollary \ref{cor-algebraic category}, we deduce
\[
\mathbb{C}[S,T] \cdot B_{G} + \mathbb{C}[S,T] \cdot (S+1) \supseteq B_{f^{\prime}FG}^{g} + \mathbb{C}[S,T] \cdot (S+1).
\]
Taking varieties finishes the proof.
\end{proof}

By Theorem 1 of \cite{SaitoArrangements}, $\V(B_{g}) \subseteq (\frac{-2d+1}{d}, 0)$, $g$ any central arrangement; by the formula \eqref{eqn-final estimation cor statement 3} for $\V(B_{g})$, the presence of roots $\frac{-2d+v}{d}$, $1 < v \leq n-1$ suggests $g$ is not free. While this is not true because $\frac{-2d+v}{d}$ might not be written in lowest terms, the following outlines how such roots can measure the distance $g$ is from being free. 

\begin{thm} \label{thm- lower bound on freeing number}
Suppose that $g$ is a central, reduced, tame hyperplane arrangement of rank $n$, $v$ an integer such that $1 < v \leq n - 1$, and $\deg(g)$ is co-prime to $v$. If $ \frac{-2 \deg(g) + v}{\deg(g)}$ is a root of the Bernstein--Sato polynomial of $g$, then $\mu_{g} \geq n - v.$
\end{thm}

\begin{proof}
Suppose $f$ is a reduced, central hyperplane arrangement such that $fg$ is free. We use the notation of the preceeding proposition and paragraphs. It suffices to prove $\deg(f) \geq n -v.$

By Proposition \ref{prop-BS ideal subset, coarser} (or Proposition 2.32 of \cite{me}) if $ \frac{-2 \deg(g) + v}{\deg(g)}$ is a root of the Bernstein--Sato polynomial of $g$ then $( \frac{-2 \deg(g) + v}{\deg(g)}, \dots,  \frac{-2 \deg(g) + v}{\deg(g)}) \in \V(B_{G})$, where $G$ corresponds to the factorization of $g$ into linear terms. By Proposition \ref{prop- freeing arrangements, embedding}, 
\[
\Delta_{S-1}( \frac{-2 \deg(g) + v}{\deg(g)}, \dots,  \frac{-2 \deg(g) + v}{\deg(g)}) \in \V(B_{fFG}^{g}) \cap \V( \mathbb{C}[S][T] \cdot (S-1)).
\]
By Theorem \ref{thm - final estimation}, there exists an indecomposable edge $X$ associated to the intersection lattice of $fg$, and an integer $j_{X}$ satisfying $0 \leq j_{X} \leq 2\deg(g_{X}) + 2 \deg(f_{X}) - 2 r(X) - \deg(f_{X})$ such that $\Delta_{S-1}( \frac{-2 \deg(g) + v}{\deg(g)}, \dots,  \frac{-2 \deg(g) + v}{\deg(g)})$ lies in the intersection of $\V( \mathbb{C}[S][T] \cdot (S-1))$ and
\[
\{\sum_{k} \deg(f_{X,k}) s_{k} + \sum_{m} \deg(g_{X,m}) t_{m} + r(X) + \deg(f_{X}) + f_{X} + j_{X} = 0 \}. 
\]
That is, 
\begin{equation} \label{eqn- freeing number thm hyperplane}
- \deg(f_{X}) + \deg(g_{X}) ( \frac{-2 \deg(g) + v}{\deg(g)}) + r(X) + \deg(f_{X}) + j_{X} = 0.
\end{equation}
Since $v$ is co-prime to $\deg(g)$, $\frac{\deg(g_{X}) v}{\deg(g)}$ can only be an integer if $\deg(g_{X}) = \deg(g)$. This implies $X = 0$ and $r(X) = n$. Rearranging \eqref{eqn- freeing number thm hyperplane} and using the upper bound on $j_{X}$ we see
\begin{equation} \label{eqn- freeing number thm inequality}
\deg(f_{X}) \geq r(X) - 2 \deg(g_{X}) + \deg(g_{X}) \frac{2 \deg(g) - v}{\deg(g)}.
\end{equation}
Because $\deg(g_{X}) = \deg(g)$ and $X = 0$, \eqref{eqn- freeing number thm inequality} simplifies to
\[
\deg(f) \geq n- v.
\]
\end{proof}

This method of argument is more versatile than the theorem suggests. In practice, information about the intersection lattice lets us drop the co-prime condition. 

\begin{example} \label{ex- freeing obstruction}
Let $g= xyzw(x+y+z)(y-z+w).$ This example is studied in \cite{DimcaSticlaruFreeExample}, Example 5.7, and \cite{SaitoFreeingExample}, Example 5.8. In the latter, Saito verifies that $\frac{-2*6 + 2}{6}$ is a root of the Bernstein--Sato polynomial. Since $\text{proj dim } \Omega^{1}(\log g) = 1$ and $n = 4$, $g$ is tame. Suppose $f$ is a central, reduced hyperplane arrangement such that $fg$ is free. Argue as in Theorem \ref{thm- lower bound on freeing number} until arriving at \eqref{eqn- freeing number thm hyperplane}. If there is an indecomposable edge $X \neq 0$ associated to the intersection lattice of $fg$ such that \eqref{eqn- freeing number thm hyperplane} holds, then $\deg(g_{X})$ must equal $3$ so that $\frac{2 \deg(g_{X})}{6}$ is an integer.  Then $g_{X}$ corresponds to the intersection of three hyperplanes of $g$; all such edges have rank $3$ (as edges of $\V(g)$). So $X$ has rank at least $3$ as an edge of the intersection lattice of $fg$. Equation \eqref{eqn- freeing number thm inequality} becomes $\deg(f_{X}) \geq 3 - 2*3 + 3* \frac{10}{6} = 2. $ On the other hand, if \eqref{eqn- freeing number thm hyperplane} is satisfied at $X = 0$, then argument of Theorem \ref{thm- lower bound on freeing number} applies and $\deg(f) \geq 2.$ Hence $\mu_{g} \geq 2.$
\end{example}

\appendix
\section{Trace of Adjoints}
Let $f$ be free and a defining equation for a divisor $Y$ at $\x$ and $f = l_{1}^{d_{1}} \cdots l_{r}^{d_{r}}$ its unique factorization into irreducibles, up to multiplication by a unit. So any reduced defining equation $f_{\red}$ for $Y$ at $\x$ is, up to multiplication by a unit, $f_{\red} = l_{1} \cdots l_{d}$. In this section we find formulae involving the commutators of $\Der_{X,\x}(-\log f)$, which by Remark \ref{rmk- log derivations}, equals $\Der_{X,\x}(-\log f_{\red})$. These formulae are crucial to the proof of Proposition \ref{prop-dual spencer terminal homology} and the precise description of the dual of $\mathscr{D}_{X,\x}[S]f^{\prime}F^{S}$. Consequently, the formulae are one of the main reasons certain Bernstein--Sato ideals have the symmetry property we used throughout the paper. These results were first proved by Castro--Jim\'enez and Ucha in Theorem 4.1.4 of \cite{JimenezUcha}; here we include a different proof. 

\begin{define}
Let $f_{\red}$ be free and $\delta_{1}, \dots , \delta_{n}$ a basis of $\Der_{X,\x}(-\log f_{\red})$. Define a matrix $\Ad_{\delta{i}}$ whose $(j, k)$ entry is $c_{k}^{i, j}$, where $c_{k}^{i,j} \in \bigO_{X,\x}$ are determined by
\[
\ad_{\delta_{i}}(\delta_{j}) = [\delta_{i}, \delta_{j}] = \sum_{k} c_{k}^{i,j} \delta_{k}.
\]
\end{define}

\begin{remark}
Note $\Ad_{\delta_{i}}$ does not determine the map $\ad_{\delta_{i}}: \Der_{X,\x}(-\log f_{\red}) \to \Der_{X,\x}(-\log f_{\red})$ since said map is not $\bigO_{X,\x}$-linear. Moreover, $\Ad_{\delta_{i}}$ depends on a choice of basis of $\Der_{X,\x}(-\log f_{\red}).$
\end{remark}

We will eventually find, given a coordinate system, a particular basis $\delta_{1}, \dots, \delta_{n}$ of $\Der_{X,\x}(-\log f_{\red})$ so that $\tr \Ad_{\delta_{i}}$, the trace of $\Ad_{\delta_{i}}$, admits a nice formula. We collect some elementary facts about the interactions between $\Der_{X,\x}(-\log f_{\red})$ and $\Omega^{\bullet}(\log f_{\red}).$ Recall by Saito, cf. 1.6 of \cite{SaitoLogarithmic}, the following: the inner product between $\Der_{X,\x}(\log f_{\red})$ and $\Omega^{1}(\log f)$ shows $\Omega^{1}(\log f_{\red})$ is the $\bigO_{X,\x}$-dual of $\Der_{X,\x}(-\log f_{\red})$; $\Omega^{\bullet}(\log f_{\red})$ is closed under taking inner products with logarithmic vector fields; $\Omega^{\bullet}(\log f_{\red})$ is closed under taking Lie derivatives along logarithmic vector fields of $f_{\red}$; if $f_{\red}$ is free then $\Omega^{k}(\log f_{\red}) = \bigwedge^{k} \Omega^{1}(\log f_{\red}).$

\begin{define} \label{def- dual basis}
For $w \in \Omega^{k}(\log f_{\red})$ and $\delta \in \Der_{X,\x}(-\log f_{\red})$ let $\iota_{\delta}(w) \in \Omega^{k-1}(\log f_{\red})$ denote the \emph{inner product} of $w$ and $\delta$. Since $f_{\red}$ is free, the induced map $\Omega^{1}(\log f_{\red}) \times \Der_{X,\x}(-\log f_{\red}) \to \bigO_{X,\x}$ is a perfect pairing. Given a basis $\delta_{1}, \dots, \delta_{n}$ of $\Der_{X,\x}(\log f_{\red})$ we may select a dual basis $\delta_{1}^{\star}, \dots, \delta_{n}^{\star}$ of $\Omega^{1}(\log f_{\red})$ such that 
\[
\iota_{\delta_{i}}(\delta_{i}^{\star}) = 1 \text{ and } \iota_{\delta_{i}}(\delta_{j}^{\star}) = 0 \text{ for $i \neq j$}.
\]
\end{define}

\begin{define}
For $w \in \Omega^{k}(\log f_{\red})$ and $\delta \in \Der_{X,\x}(-\log f_{\red})$ let $L_{\delta_{i}}(w) \in \Omega^{k}(\log f_{\red})$ denote the \emph{Lie derivative} of $w$ along $\delta_{i}$. Let $\delta_{1}, \dots, \delta_{n}$ and $\delta_{1}^{\star}, \dots, \delta_{n}^{\star}$ be as in Definition \ref{def- dual basis}. Then there exists a unique choice of $b_{k}^{i,j} \in \bigO_{X,\x}$ such that
\[
L_{\delta_{i}}(\delta_{j}^{\star}) = \sum_{k} b_{k}^{i, j} \delta_{k}^{\star}.
\]
Define the matrix $\Lie_{\delta_{i}}$ to have $(j, k)$ entry $b_{k}^{i,j}.$
\end{define}

\begin{remark}
Just like $\Ad_{\delta_{i}}$, the matrix $\Lie_{\delta_{i}}$ does not determine the map $L_{\delta_{i}}: \Omega^{1}(\log f_{\red}) \to \Omega^{1}(\log f_{\red})$; moreover, $\Lie_{\delta_{i}}$ depends on the choice of basis $\delta_{1}, \dots, \delta_{n}$ of $\Der_{X,\x}(-\log f)$ which in turn determines the basis $\delta_{1}^{\star}, \dots, \delta_{n}^{\star}$ of $\Omega^{1}(\log f)$.
\end{remark}

We need the following elementary lemma. It is well known for vector fields and differential forms and can easily be shown to hold in the logarithmic case by writing a logarithmic differential form as $\frac{1}{f_{\red}} w$ where $w$ is a differential form. 

\begin{lemma} \label{lemma- lie contraction commutator}
Let $X, Y \in \Der_{X,\x}(\log f_{\red})$. Then as maps from $\Omega^{k}(\log f_{\red}) \to \Omega^{k-1}(\log f_{\red})$, we have
\[
\iota_{[X,Y]} = [L_{X}, \iota_{Y}].
\]
\end{lemma}

\begin{prop} \label{prop- ad equals lie transpose}
If $f_{\red}$ is free and $\delta_{1}, \dots, \delta_{n}$ is a basis for $\Der_{X,\x}(-\log f_{\red})$, then
\[
\Ad_{\delta_{i}} = - \Lie_{\delta_{i}}^{T}.
\]
\end{prop}

\begin{proof}
On one hand, 
\[
\iota_{\ad_{\delta_{i}}(\delta_{j})}(\delta_{t}^{\star}) = \iota_{\sum_{k} c_{k}^{i,j} \delta_{k}} (\delta_{t}^{\star}) = c_{t}^{i,j}.
\]
On the other hand,
\[
[L_{\delta_{i}}, \iota_{\delta_{j}}](\delta_{t}^{\star}) = - \iota_{\delta_{j}}(L_{\delta_{i}}(\delta_{t}^{\star})) = - \iota_{\delta_{j}}(\sum_{k} b_{k}^{i, t} \delta_{k}^{\star}) = - b_{j}^{i,t},
\]
as the Lie derivative of a vector field on a constant is zero. Now use Lemma \ref{lemma- lie contraction commutator}.
\end{proof}

Since $f_{\red}$ is free, $\Omega^{n}(\log f_{\red})$ is a free, cyclic $\bigO_{X,\x}$-module generated by $\delta_{1}^{\star} \wedge \cdots \wedge \delta_{n}^{\star}.$ Moreover:

\begin{prop} \label{prop- Lie deriv top form}
Let $f_{\red}$ be free and $\delta_{1}, \dots, \delta_{n}$ be a basis for $\Der_{X,\x}(-\log f_{\red})$. Then 
\[
L_{\delta_{i}} (\delta_{1}^{\star} \wedge \cdots \wedge \delta_{n}^{\star}) = - \tr \Ad_{\delta_{i}} (\delta_{1}^{\star} \wedge \cdots \wedge \delta_{n}^{\star}).
\]
\end{prop}

\begin{proof}

By basic facts of Lie derivatives:
\begin{align*}
    L_{\delta_{i}}(\delta_{1}^{\star} \wedge \cdots \wedge \delta_{n}^{\star}) 
        &= \sum_{j} \delta_{1}^{\star} \wedge \cdots \wedge \delta_{j-1}^{\star} \wedge L_{\delta_{i}}(\delta_{j}^{\star}) \wedge \delta_{j+1}^{\star} \wedge \cdots \wedge \delta_{n}^{\star} \\
        &= \sum_{j} \delta_{1}^{\star} \wedge \cdots \wedge \delta_{j-1}^{\star} \wedge (\sum_{k} b_{k}^{i, j} \delta_{k}^{\star}) \wedge \delta_{j+1}^{\star} \wedge \cdots \wedge \delta_{n}^{\star} \\
        &= (\sum_{k} b_{k}^{i, k} ) (\delta_{1}^{\star} \wedge \cdots \wedge \delta_{n}^{\star}).
\end{align*}
The result follows by Proposition \ref{prop- ad equals lie transpose}.
\end{proof}

We will also need the following standard definition and proposition from differential geometry.

\begin{define}
Consider local coordinates $x_{1}, \dots, x_{n}$. Let $\delta$ be a vector field. Then $\Div(\delta)$ is the \emph{divergence} of $\delta$ with respect to the $n$-form $dx_{1} \wedge \cdots \wedge dx_{n}$ and is defined by:
\[
L_{\delta}(dx_{1} \wedge \cdots \wedge dx_{n}) = \Div(\delta) (dx_{1} \wedge \cdots \wedge dx_{n}).
\]
\end{define}

\begin{prop} \label{prop-div right-left formula}

In local coordinates $x_{1}, \cdots, x_{n}$, write the vector field $\delta$ as $\delta = \sum_{k} h_{k} \frac{\partial}{\partial_{x_{k}}}$, where $h_{k} \in \bigO_{X,\x}$. Then $\Div(\delta)$ with respect to $dx_{1} \wedge \cdots \wedge dx_{n}$ satisfies the formula
\[
\Div(\delta) = \sum_{k} \frac{\partial}{\partial_{x_{k}}} \bullet h_{k}.
\]

\end{prop}

\begin{proof}

Write $dx = dx_{1} \wedge \cdots \wedge dx_{n}$. By Cartan's formula, $L_{\delta}(dx) = d(\iota_{\delta} (dx)).$ Using the skew-symmetric properties of the inner product we deduce:
\begin{align*}
d(\iota_{\delta}(dx)) 
    &= d( \sum_{k} (-1)^{k-1} (dx_{1} \wedge \cdots \wedge \iota_{\delta}(dx_{k}) \wedge \cdots \wedge dx_{n})) \\
    &= d( \sum_{k} (-1)^{k-1} h_{k} (dx_{1} \wedge \cdots \wedge \widehat{dx_{k}} \wedge \cdots \wedge dx_{n})) \\
    &= (\sum_{k} \frac{\partial}{\partial_{x_{k}}} \bullet h_{k}) dx.
\end{align*}
\end{proof}

Consider a basis $\delta_{1}, \dots, \delta_{n}$ of $\Der_{X,\x}(-\log f_{\red})$. Then for any choice of coordinates $x_{1}, \dots, x_{n}$, there exists a corresponding unit $u \in \mathscr{O}_{X,\x}$ such that $\delta_{1}^{\star} \wedge \cdots \wedge \delta_{n}^{\star} = \frac{u}{f_{\red}} dx_{1} \wedge \cdots \wedge dx_{n}$. See the proof of the first theorem on page 270 of \cite{SaitoLogarithmic} for justification. Clearly $u \delta_{1}, \dots, \delta_{n}$ is still a basis of $\Der_{X,\x}(-\log f_{\red})$ and since $\frac{1}{u} \delta_{1}^{\star} = (u \delta_{1})^{\star}$, the logarithmic forms $(u \delta_{1})^{\star}, \delta_{2}^{\star}, \dots, \delta_{n}^{\star}$ constitute a dual basis of $\Omega^{1}(\log f_{\red})$ satisfying:
\[
(u \delta_{1})^{\star} \wedge \delta_{2}^{\star} \wedge \cdots \wedge \delta_{n}^{\star} = \frac{1}{f_{\red}} dx_{1} \wedge \cdots \wedge dx_{n}.
\]
This shows, as long as $n \geq 2$, that one can always find a basis of $\Der_{X,\x}(-\log f_{\red})$ satisfying the conditions of the following definition:

\begin{define} \label{def-preferred basis}

Let $f_{\red}$ have Euler homogeneity $E$ at $\x$. Having fixed a coordinate system $x_{1}, \dots, x_{n}$, consider a basis $\delta_{1}, \dots, \delta_{n}$ of $\Der_{X,\x}(-\log f_{\red})$ such that $\delta_{n} = E$ and $\delta_{1}, \dots, \delta_{n-1}$ is a basis of $\Der_{X,\x}(-\log_{0} f_{\red})$. Such a basis is a \emph{preferred basis} of $\Der_{X,\x}(-\log f_{\red})$ if, in addition, 
\[
\delta_{1}^{\star} \wedge \cdots \wedge \delta_{n}^{\star} = \frac{1}{f_{\red}} dx_{1} \wedge \cdots \wedge dx_{n}.
\]
\end{define}

We are finally ready to state the main formula of this section.

\begin{prop} \label{prop: main adjoint right-left formula}

Let $f_{\red}$ be free with Euler homogeneity $E$. Given a coordinate system $x_{1}, \dots, x_{n}$, let $\delta_{1}, \dots, \delta_{n}$ be a preferred basis of $\Der_{X,x}(-\log f_{\red})$. Write $\delta_{i} = \sum_{k} h_{k, i} \frac{\partial}{\partial_{x_{k}}}.$ Then
\begin{enumerate}[(i)]
    \item $\tr \Ad_{\delta_{i}} = - \sum_{k} \frac{\partial}{\partial_{z_{k}}} \bullet h_{k, i}$ for $i \neq n$;
    \item $\tr \Ad_{\delta_{n}} = - \sum_{k} \frac{\partial}{\partial_{z_{k}}} \bullet h_{k, n} + 1$.
\end{enumerate}
 
\end{prop}

\begin{proof}

Write $dx = dx_{1} \wedge \cdots \wedge dx_{n}$. Because $\delta_{1}, \cdots, \delta_{n}$ is a preferred basis of $\Der_{X,\x}(-\log f_{\red})$ and by standard properties of the Lie derivative
\begin{align} \label{eqn-main trace formula}
L_{\delta_{i}}(\delta_{1}^{\star} \wedge \cdots \wedge \delta_{n}^{\star}) 
    &= L_{\delta_{i}}(\frac{1}{f_{\red}} dx) = L_{\delta_{i}}(\frac{1}{f_{\red}})dx + \frac{1}{f_{\red}} L_{\delta_{i}} (dx) \\
    &= L_{\delta_{i}}(\frac{1}{f_{\red}})dx + (\frac{1}{f_{\red}} \sum_{k} \frac{\partial}{\partial_{x_{k}}} \bullet h_{k,i}) dx. \nonumber
\end{align}
Note that the last equality of \eqref{eqn-main trace formula} follows by Proposition \ref{prop-div right-left formula}. When $i \neq n$, $L_{\delta_{i}}(\frac{1}{f_{\red}}) = 0$; when $i = n$, $L_{\delta_{n}}(\frac{1}{f_{\red}}) = - \frac{1}{f_{\red}}.$ The result follows by the definition of a preferred basis together with Proposition \ref{prop- Lie deriv top form}. 
\end{proof}

\section{Budur's Conjecture for Central, Reduced, Free Arrangements}

In \cite{BudurLocalSystems}, Budur conjectured that exponentiating $\V(B_{F,\x})$ (here $F = (f_{1}, \dots, f_{r})$ is collection of polynomials) gives the support of the Sabbah specialization functor, generalizing the fact that exponentiating the roots of the Bernstein--Sato polynomial gives the support of the nearby cycle functor, cf. Conjecture 2 of loc. cit. In the same paper he reduced this conjecture to proving, in language we will shortly define, that if $A-1 \in \V(B_{F,\x})$ then a certain $\D_{X,\x}$-linear map $\nabla_{A}$ is not surjective, cf. Proposition 2 of loc. cit. For $f = f_{1} \cdots f_{r}$ a central, reduced, and free hyperplane arrangement and $F = (f_{1}, \dots, f_{r})$ an arbitrary factorization of $f$ we provide a proof here. Theorem 3.5.3 of the recently announced paper \cite{BudurConjecture} gives a general proof of the conjecture by proving the claim about $\nabla_{A}$ for general points in the codimension one components of $\V(B_{F,\x})$. 
Our method relies on the computation of $\V(B_{F,0})$ given in Corollary \ref{cor-final estimation cor} and the behavior of $\nabla_{A}$ under duality, cf. Section 4 of \cite{me}.

First, let us clarify our terminology. (See also Section 3 of \cite{me} for more details). For $a_{1}, \dots, a_{r} \in \mathbb{C}$, denote by $S-A$ the sequence $s_{1} - a_{1}, \dots, s_{r} - a_{r}$. Similarly, let $A$ and $A-1$ denote the tuple $a_{1}, \dots, a_{r}$ and $a_{1}-1, \dots, a_{r}-1$ respectively. There is an injective $\D_{X,\x}$-linear map $\nabla: \D_{X,\x}[S]F^{S} \to \D_{X,\x}[S]F^{S}$ given by sending every $s_{k}$ to $s_{k} + 1$ and identifying $F^{S+1}$ with  $f F^{S}$. This induces the $\D_{X,\x}$-linear map 
\[
\nabla_{A}: \frac{\D_{X,\x}[S]F^{S}}{(S-A)\D_{X,\x}[S]F^{S}} \to \frac{\D_{X,\x}[S]F^{S}}{(S-(A-1))\D_{X,\x}[S]F^{S}}.
\]

By Proposition 2 of \cite{BudurLocalSystems}, to prove Budur's conjecture in our setting, it suffices to prove the following:

\begin{prop} \label{prop - appendix b prop}
Let $f=f_{1} \cdots f_{r}$ be a central, reduced, and free hyperplane arrangement where the $f_{k}$ are not necessarily linear forms. Let $F = (f_{1}, \dots, f_{r})$. If $A-1 \in \V(B_{F,0})$, then 
\[
\nabla_{A}: \frac{\D_{\mathbb{C}^{n},0}[S] F^{S}}{(S-A) \D_{\mathbb{C}^{n},0}[S]F^{S}} \to \frac{\D_{\mathbb{C}^{n},0}[S] F^{S}}{(S-(A-1)) \D_{\mathbb{C}^{n},0}[S]F^{S}}
\] is not surjective.
\end{prop}

\begin{proof}
Since the $f_{k}$ are globally defined we may consider the global version of $\nabla_{A}$. Since $f$ is central, there is a natural $\mathbb{C}^{\star}$-action on $\V(f)$; moreover, $\nabla_{A}$ is equivariant with respect to this action. Therefore $\nabla_{A}$ is surjective at $0$ if and only if it is surjective at all $\x \in \V(f)$. So it suffices to prove $\nabla_{A}$ is not surjective for 
\[
A-1 \in \bigcup_{j = 0}^{2d - 2n} \{ \left( \sum d_{k} s_{k} \right) + n + j= 0\},
\]
when $f$ is indecomposable of rank $n$ and degree $d$, cf. Corollary \ref{cor-final estimation cor} and Remark \ref{remark-local to global b-poly}.

Since $f$ is reduced, $\V(B_{F,0})$ is invariant under the map $\varphi$ on $\mathbb{C}[S]$ induced by $s_{k} \mapsto -s_{k} - 2$, cf. Theorem \ref{thm-unmixed symmetry} or Proposition 8 of \cite{MaisonobeFree}. This map sends $\{ \left( \sum d_{k} s_{k} \right) + n + j= 0\}$ to $\{ \left( \sum d_{k} s_{k} \right) +n + (2d - 2n - j)= 0\}$. Theorem 4.18 and Theorem 4.19 of \cite{me} prove that the invariance of $\varphi$ forces $\nabla_{A}$ to be surjective if and only if $\nabla_{-A}$ is surjective. So if we show $\nabla_{A}$ is not surjective for all $A-1 \in \{ \left( \sum d_{k} s_{k} \right) + n + j= 0\}$ then we will have also shown $\nabla_{-A}$ is not surjective for all $-A - 1 \in \{ \left( \sum d_{k} s_{k} \right) + 2d - n - j= 0\}$. Thus it suffices to prove $\nabla_{A}$ is not surjective for 
\[
A-1 \in \bigcup_{j = 0}^{d - n} \{ \left( \sum d_{k} s_{k} \right) + n + j= 0\}.
\]

Let $f^{\prime}$ divide $f$, where the degree $d^{\prime}$ of $f^{\prime}$ is less than $d$. Just as $\nabla_{A}$ is induced by the $\D_{\mathbb{C}^{n},0}$-injection $\nabla: \D_{\mathbb{C},0}[S]F^{S} \to \D_{\mathbb{C}^{n},0} F^{S}$ sending each $s_{k}$ to $s_{k} + 1$, there is an induced $\D_{\mathbb{C}^{n},0}$-map 
\[
\nabla_{A}^{f^{\prime}} : \frac{\D_{\mathbb{C}^{n},0}[S] F^{S}}{(S-A) \D_{\mathbb{C}^{n},0}[S]F^{S}} \to \frac{\D_{\mathbb{C}^{n},0}[S] f^{\prime}F^{S}}{(S-(A-1)) \D_{\mathbb{C}^{n},0}[S]f^{\prime}F^{S}}.
\]
Moreover, the non-injectivity of $\nabla_{A}^{f^{\prime}}$ implies the non-injectivity of $\nabla_{A}$. Arguing as in Section 3 of \cite{me}, we can prove a version of Theorem 3.11 of loc. cit. for $\nabla_{A}^{f^{\prime}}$: if $\nabla_{A}^{f^{\prime}}$ is injective, then it is surjective. By Theorem 4.19 of loc. cit., it thus suffices to prove $\nabla_{A}^{f^{\prime}}$ is not surjective for 
\[
A-1 \in \{ \left( \sum d_{k}s_{k} \right) + n + d^{\prime} = 0 \}.
\]

Now we are in the situation of Theorem \ref{thm-first subset thm}, where instead of looking for $vB(S) \in \ann_{\D_{\mathbb{C}^{n},0}[S]}f^{\prime}F^{S} + \D_{\mathbb{C}^{n},0}[S] \cdot g$, where $g = \frac{f}{f^{\prime}}$, we are considering the following possibility: 
\begin{equation} \label{eqn-appendix b finding 1}
1 \in \ann_{\D_{\mathbb{C}^{n},0}[S]}f^{\prime}F^{S} + \D_{\mathbb{C}^{n},0}[S] \cdot g + (S-(A-1)) \D_{\mathbb{C}^{n},0}[S].
\end{equation}
Suppose, towards contradiction, \eqref{eqn-appendix b finding 1} holds, i.e. $\nabla_{A}^{f^{\prime}}$ is surjective. We argue as in Theorem \ref{thm-first subset thm}, except letting $B(S)$ and $v$ be $1$, and obtain an equation resembling \eqref{eqn-third eqn main thm star} except with additional terms on the right hand side from $(S-(A-1))\D_{\mathbb{C}^{n},0}[S]$. Look at the right constant terms of this version of \eqref{eqn-third eqn main thm star}, evaluate each $s_{k}$ at $a_{k}-1$, and regard every summand as a power series. This gives an equality of elements in $\mathscr{O}_{X,0}$; denote by $\mathfrak{m}_{0}$ the maximal ideal of $\mathscr{O}_{X,0}$. By the argument of Theorem \ref{thm-first subset thm}, the only piece of the right hand side outside of $\mathfrak{m}_{0}$ can come from $L_{\textbf{0}} g$ as the relevant pieces from $P \psi_{f^{\prime}F,0}(E)$ and the $(S-(A-1)\D_{\mathbb{C}^{n},0}[S]$ terms vanished after sending each $s_{k}$ to $a_{k}-1$ and there are no such pieces from the $Q_{j} \psi_{f^{\prime}F}(\delta_{j})$ terms by Lemma \ref{lemma-constant term syzygy}. Certainly $g \in \mathfrak{m}_{0}$. Thus the entire right hand side lies in $\mathfrak{m}_{0}$.  Since $1 \notin \mathfrak{m}_{0}$, our assumption that \eqref{eqn-appendix b finding 1} holds is actually impossible, and the claim is proved. 
\end{proof}

\begin{remark} \label{rmk- appendix b remark}
\begin{enumerate}[(a)]
    \item One can argue similarly for non-reduced $f$ if we assume $F$ is unmixed up to units and we check Theorem 4.18 and Theorem 4.19 of \cite{me} for $F$ unmixed up to units. In particular, this applies when $F$ is a factorization into linear terms. We leave this to the reader.
    \item In this case, we obtain the expected formula \eqref{eqn-final estimation cor statement 3} for the roots of Bernstein--Sato polynomial of an appropriate $f$ by Remark \ref{rmk- appendix b remark}.(a) and the strategy outlined in Remark \ref{rmk- gesturing at Budur improvement}.(a). This approach does not rely on \cite{BudurConjecture}.
    \item The primary purpose of Theorem 3.5.3 of \cite{BudurConjecture} is to analyze $\text{Exp}(\V(B_{F,0})).$ When $f$ is simply a central, reduced hyperplane arrangement and $L$ is a factorization of $f$ into linear forms, $\text{Exp}(\V(B_{L,0})$ can be explicitly computed by Theorem \ref{thm- first supset thm} (or Maisonobe's Proposition 10 of \cite{MaisonobeFree}) and Corollary 2 of \cite{BudurLocalSystems}. In this case, Budur's conjecture holds without appeal to \cite{BudurConjecture}. Similar approaches work for non-reduced $f$ and different factorizations $F$ of $f$, cf. Corollary \ref{cor-main BS member any factorization} and also Remark 6.10 of \cite{BudurLocalSystems}.
\end{enumerate}
\end{remark}

\bibliographystyle{abbrv}

\end{document}